%% file: RealStructureMatroids.tex
\documentclass{lms}

\input{commands}

\title[Real phase structures on matroid fans]{Real phase structures on matroid fans and matroid orientations}

\author{Johannes Rau, Arthur Renaudineau and Kris Shaw}
\classno{52C40 (Primary), 05B35, 14T05 (Secondary)}
\extraline{J.R. acknowledges support from the FAPA grant by the Facultad de Ciencias, Universidad de los Andes. A.R. acknowledges support from the Labex CEMPI (ANR-11-LABX-0007-01). K.S. acknowledges support from the Trond Mohn Foundation project ``Algebraic and Topological Cycles in Complex and Tropical Geometries".}

\begin{document}
\maketitle

\begin{abstract}
We introduce the notion of real phase structure on rational polyhedral fans in Euclidean space. Such a structure consists of an assignment of affine spaces over $\Z/2\Z$ to each top dimensional face of the fan subject to two conditions. 

Given an oriented matroid we can construct a real phase structure on the fan of the underlying matroid. Conversely, we show 
that from a real phase structure on a matroid fan we can produce an orientation of the underlying matroid.
Thus real phase structures 
are cryptomorphic to matroid orientations. 

The topes of the orientated matroid are recovered immediately from the real phase structure. We also provide a direct way to recover the signed circuits of the oriented matroid from the real phase structure. 

\end{abstract}

\tableofcontents{}

\input{introduction}

\input{equivalence}

\bibliographystyle{alpha}
\bibliography{biblio}
\affiliationone{
   Johannes Rau\\
   Universidad de los Andes,\\ Carrera 1 no.~18A-12,\\ Bogotá, Colombia.
   \email{j.rau@uniandes.edu.co}}
\affiliationtwo{
   Arthur Renaudineau\\
   Univ. Lille, CNRS, UMR 8524 -\\
Laboratoire Paul Painlev\'e\\ F-59000 Lille\\ France.
   \email{arthur.renaudineau@univ-lille.fr}}
\affiliationthree{Kris Shaw\\
Department of Mathematics \\
University of Oslo\\
 Blindern, 0316\\
 Oslo, Norway
 \email{krisshaw@math.uio.no}} 

\end{document}

%% file: commands.tex
\pdfoutput=1

\usepackage[utf8]{inputenc} 
\usepackage[T1]{fontenc} 
\usepackage{mathtools}
\usepackage{amsmath}
\usepackage{amssymb}
\usepackage{hyperref}
\usepackage{url}
\usepackage{babel}
\usepackage{pgf,tikz,pgfplots}
\usepackage{graphicx}
\usepackage{tikz-cd}
\usepackage{tikz-3dplot}
\usepackage{csquotes} 

\usepackage{tikz, float} \usetikzlibrary {positioning}
\usetikzlibrary{shapes.misc}
\usetikzlibrary{arrows}

\usetikzlibrary{calc}
\usepackage{amsfonts}
\input xy
\xyoption{all}

\newcommand{\kris}[1]{{\textcolor{red}{#1}}}
\newcommand{\arthur}[1]{{\textcolor{blue}{#1}}}
\newcommand{\jojo}[1]{{\textcolor{green}{#1}}}

\newcommand{\Z}{{\mathbb Z}}

\newcommand{\PP}{{\mathbb P}}

\newcommand{\R}{{\mathbb R}}

\newcommand{\CC}{{\mathcal C}}

\newcommand{\affFan}{\Sigma}
\newcommand{\projFan}{\PP\Sigma}

\newcommand{\Facets}{\text{Facets}}
\newcommand{\EEE}{{\mathcal E}}
\newcommand{\TTT}{{\mathcal T}}
\newcommand{\CCC}{{\mathcal C}}
\newcommand{\FFF}{{\mathcal F}}
\newcommand{\LLL}{{\mathcal L}}

\newcommand{\MMM}{{\mathcal M}}

\DeclareMathOperator{\cl}{cl}

\newcommand{\E}{{\mathcal E}}
\newcommand{\F}{{\mathcal F}}

\newcommand{\M}{{\mathcal M}}

\renewcommand{\S}{{\mathcal S}}

\renewcommand{\M}{{\mathcal M}}
\renewcommand{\L}{{\mathcal L}}

\newcommand{\Aff}{\text{Aff}}

\DeclareMathOperator{\rk}{r}

\newtheorem{thm}{Theorem}[section] 
\newtheorem{lemma}[thm]{Lemma}     

\newtheorem{proposition}[thm]{Proposition}

\newnumbered{assertion}{Assertion}    
\newnumbered{conjecture}{Conjecture}  
\newnumbered{definition}{Definition}
\newnumbered{hypothesis}{Hypothesis}
\newnumbered{remark}{Remark}
\newnumbered{note}{Note}
\newnumbered{observation}{Observation}
\newnumbered{problem}{Problem}
\newnumbered{ques}{Question}
\newnumbered{algorithm}{Algorithm}
\newnumbered{example}{Example}
\newunnumbered{notation}{Notation} %

\tikzset{%
  add/.style args={#1 and #2}{to path={%
 ($(\tikztostart)!-#1!(\tikztotarget)$)--($(\tikztotarget)!-#2!(\tikztostart)$)%
  \tikztonodes}}
} 

\newcommand{\comment}[1]{}

%% file: introduction.tex
\section{Introduction}

We propose a definition of real phase structures on rational polyhedral fans, with specific attention to matroid fans, 
also known as Bergman fans of matroids. 
A real phase structure on a fan in $\R^n$ is the specification of an affine subspace of $(\Z/2\Z)^n$ for each top-dimensional cone
of the fan and subject to two conditions, see Definition \ref{def:realstructure}. 
Given an oriented matroid $\MMM$, its underlying non-oriented matroid is denoted by $\underline{\MMM}$. For a fixed matroid $M$, 
an oriented matroid $\MMM$ such that $M = \underline{\MMM}$ is called an orientation of $M$. 
Our main theorem states that a  real phase structure  on the associated matroid fan $\Sigma_M$ 
is equivalent to an orientation of the underlying matroid $M$.

\begin{thm}\label{thm:orientedphase} 
Given a fixed matroid $M$, there is a natural bijection between orientations $\MMM$ of $M$ and real phase structures on the matroid fan $\Sigma_M$.
In other words, oriented matroids and real phase structures on matroid fans are cryptomorphic concepts.
\end{thm}

Real phase structures have previously been defined on tropical curves  \cite[Section 7.2]{GrishaICM}, \cite{Bertrand},  \cite{BBR}, and on non-singular tropical hypersurfaces \cite{Renaudineau17} 
and \cite{RS}. These two special cases of real phase structures on polyhedral complexes have been used in the study of real enumerative geometry, in the form of Welschinger invariants, and the topology of real algebraic varieties.
Studying the fans of matroids  has led to major breakthroughs in understanding the behaviour of many matroid invariants \cite{AdiprasitoHuhKatz}, \cite{ArdilaDenhamHuh}.  Our hope is that matroid fans equipped with real phase structures will  find similar applications in the study of oriented matroids.

The initial goal of our investigation into real phase structures on matroid fans was to generalise the spectral sequence from \cite{RS} and subsequent bounds obtained  on Betti numbers of real algebraic hypersurfaces arising from Viro's patchworking procedure  \cite{Viro6and7}, \cite{ViroPatch} to more general spaces. This is the subject of a forthcoming paper, in which we define the real part of a tropical variety equipped with a real phase structure. This is similar to the patchworking of tropical linear spaces of Celaya, Loho and Yuen \cite{CLY}, yet we do not require the underlying matroid of the oriented matroid to be uniform.

We construct the bijection from Theorem \ref{thm:orientedphase} explicitly by 
assigning to each oriented matroid $\MMM$ with $\underline{\MMM} = M$ a real phase structure on $\Sigma_M$. 
The real part of this real phase structure produces a topological representation of the 
oriented matroid in the sense of the famous theorem of Folkman and Lawrence \cite{FolkmanLawrence}. 
This recovers similar constructions by Ardila, Klivans, and Williams  \cite{AKW} and Celaya \cite{Celaya} 
for the positive real part and real part of a matroid fan. 
In the forthcoming paper,  we combine this fact with Theorem \ref{thm:orientedphase} 
to prove that the real part of a non-singular tropical variety is a PL-manifold.

The  paper is organised as follows. 
In Section \ref{sec:fans}, we recall the definition of rational polyhedral fans and introduce the notion of real phase structures on them. Then we specialise to the case of fans arising from matroids. In Subsection \ref{subsecMatroids}, we recall the construction of Ardila and Klivans of a matroid fan from its lattice of flats. In Subsection \ref{sec:necklace}, we introduce the notion of necklace line arrangements and translate one of the conditions for real phase structures into this language in the case of matroid fans.
We show that real phase structures on matroid fans behave well under the operations of deletion and contraction of the underlying matroid in Subsection \ref{subsec:deleteandcontract}.

Section \ref{sec:Orientedmatrealphase} contains our main results. First, in subsection \ref{subsec: orientedmatroids} we show how to obtain a real phase structure from an oriented matroid. 
Then in  Subsections \ref{subsec:orientedmatroidquotients} and \ref{subsec:equiv}, we prove that every real phase structure on a matroid fan is obtained in this fashion by considering oriented matroid quotients. Finally, for convenience,  in Subsection \ref{RealStructureToSignature} we give an explicit description of how to recover the signed circuits of the oriented matroid from the real phase structure on a matroid fan. 

\section*{Acknowledgement}
We thank Marcel Celaya, Georg Loho, Felipe Rincón, and Chi Ho Yuen for enlightening conversations. We are also grateful to Jerónimo Valencia and two anonymous referees for helpful comments on preliminary versions of this paper.

%% file: equivalence.tex
\section{Real phase structures}
\subsection{Fans and real phase structures}\label{sec:fans}
A polyhedral fan $\Sigma$ in $\R^n$ is a collection of convex polyhedral cones such that every face of a cone in $\Sigma$ is also a cone in $\Sigma$, and the intersection of two cones is a face of both of them. A fan is rational if all of its cones are generated over $\Z$.
The faces of $\Sigma$ which are maximal with respect to inclusion are called the facets of $\Sigma$. We denote the set of facets of $\Sigma$ by $\Facets(\Sigma)$. 
A rational polyhedral fan is pure dimensional if its facets are all of  equal dimension. Here we only consider fans of pure dimension.

Throughout $V$ will be a vector space over $\Z_2 := \Z / 2\Z$. We  let $\Aff_d(V)$ denote the set of all affine subspaces of dimension $d$ in $V$. 
Given $A \in \Aff_d(V)$, we denote by $T(A)$ the \emph{tangent space} of $A$. In particular, 
the space $T(A)$ is 
the $d$-dimensional linear subspace in $V$ generated by the vectors $x-y$ for  $x,y \in A$. 
If $\sigma$ is a rational polyhedral cone in $\R^n$, we denote its tangent space by  $T(\sigma)$. The space $T(\sigma)$ is the linear span over $\R$ of the vectors generating the cone. The set of integer points in  the  tangent space of $\sigma$ is  denoted by $T_{\Z}(\sigma)$, and the reduction  mod $2$ of these points is denoted by $T_{\Z_2}(\sigma)$.

\begin{definition}\label{def:evencovering} 
A collection of subsets  of a set such that every element in the union is contained in an even number of the subsets is called an \emph{even covering}. 
\end{definition}

\begin{definition}\label{def:realstructure}
Let $\Sigma$ be a rational polyhedral fan of pure dimension $d$ in $\R^n$.
A \emph{real phase structure} $\mathcal{E}$ on $\Sigma$ is a map
$$ \mathcal{E} \colon \Facets(\Sigma) \to \Aff_d(\Z_2^n)$$
such that 
\begin{enumerate}
\item for every facet $\sigma$ of $\Sigma$, the set $ \mathcal{E}(\sigma) $ is an affine subspace of $\Z_2^n$ parallel to $\sigma$, in formulas, $T(\EEE(\sigma)) = T_{\Z_2}(\sigma)$; 
\item \label{cond2} for every codimension one face $\tau$ of $\Sigma$ with facets $\sigma_1, \dots , \sigma_k$ adjacent to it, the sets $\E(\sigma_1), \dots, \E(\sigma_k)$ are an even covering. 
\end{enumerate}
\end{definition}

\begin{definition}\label{def:reorientation}
Let $\EEE$ be a  real phase structure on $\Sigma$. A \emph{reorientation of $\EEE$} is a real phase structure
 $\EEE'$ obtained by translating all affine subspaces in a real phase structure  $\EEE$ by a fixed vector $\varepsilon \in \Z_2^n$. 
 In other words $\EEE'(\sigma) = \EEE(\sigma) + \varepsilon$ for all $\sigma \in \Sigma$. 
\end{definition}

\subsection{Matroid fans} \label{subsecMatroids}

A matroid $M$ is a finite set $E$ together with a function $\rk  : 2^E \to \mathbb{N}_{\geq0}$, where $2^E$ denotes the power set of $E$. 
The set $E$ is called the ground set of $M$ and $r$ the rank function. The rank function is subject to the axioms: 
\begin{enumerate}
\item $0 \leq \rk(A) \leq |A|$ for all $A \subseteq E$
\item if $A \subseteq B \subseteq E$, then $\rk(A) \leq \rk(B)$
\item if $A, B \subseteq E$,  then $\rk(A \cup B) +  \rk(A \cap B)  \leq \rk(A ) + \rk(B) $.  
\end{enumerate}

The rank function defines a closure operator on subsets by $$\cl(A) = \{ i \in E \ | \  \rk(A) = \rk(A \cup i) \} \supseteq A.$$ A subset $F \subseteq E$ is a flat of $M$ if it is closed with
respect to this operator, namely $\cl(F) = F$. The flats of a matroid $M$ ordered by inclusion form a lattice, known as the lattice of flats, which we denote by $\mathcal{L}$. 

A loop is an element of the ground set for which $\rk(i) = 0$. Parallel elements are pairs of non-loop elements $i, j \in E$ for which $\rk(ij) = 1$. A matroid is simple if it contains no loops or parallel elements.
A circuit is any set $A\subset E$ such that $|A|=\rk(A)+1$ and 
$|A| = \rk(A\setminus i)$ for any $i\in A$.
A coloop is an element that does not belong to any circuit.

Given a loopfree matroid $M$ on the base set $E$, we denote by $\Sigma_M$ the \emph{affine} matroid fan in
$\R^E$ and by $\projFan_M = \Sigma_M/\langle (1,\dots,1) \rangle$ the \emph{projective} matroid fan in $\R^E / \langle (1,\dots,1) \rangle $. 
We now describe how to construct both of these fans following Ardila and Klivans \cite{ArdilaKlivans}. 
Note that Ardila and Klivans use the terminology Bergman fan of a matroid $M$.
Fix the vectors $v_i = -e_i$ where $\{e_1, \dots, e_n\}$ is the standard basis of $\R^E$ for $E = \{1, \dots, n\}$ and set $v_I = \sum_{i \in I} v_i$ for any subset $I \subset E$.
For a 
chain of flats $$\F =\{ \emptyset \subsetneq F_1 \subsetneq F_2 \subsetneq \dots \subsetneq F_k \subsetneq E\}$$
in the lattice of flats $\L$,
define the $(k+1)$-dimensional cone 
\[\sigma_{\F} = \langle v_{F_1}, \dots , v_{F_k},  v_E , - v_E\rangle_{\geq 0}.\]
The affine  matroid fan $\Sigma_M$ is the collection of all such cones ranging over the chains in $\L$. 
In particular, the top dimensional faces of  $\Sigma_M$ are in one to one correspondence with 
the maximal chains in the lattice of flats of $M$. 
The projective matroid fan $\projFan_M$ is the image of $\Sigma_M$ in the quotient $\R^E / (1, \dots, 1)$. 

If a matroid $M$ has loops $L = \cl(\emptyset)$, then we set $\Sigma_M := \Sigma_{M/L} \subset \R^{E \setminus L}$ and 
$\projFan_M := \projFan_{M/L} \subset \R^{E \setminus L}/(1,\dots,1)$. 
This is a practical definition, the more coherent point of view is to regard 
the affine and projective matroid fans as subsets of  boundary strata of tropical affine space  and tropical projective space, respectively. See \cite{ShawIntMat} or \cite{MikRau} for more details.

\begin{figure}
\includegraphics[scale=0.25]{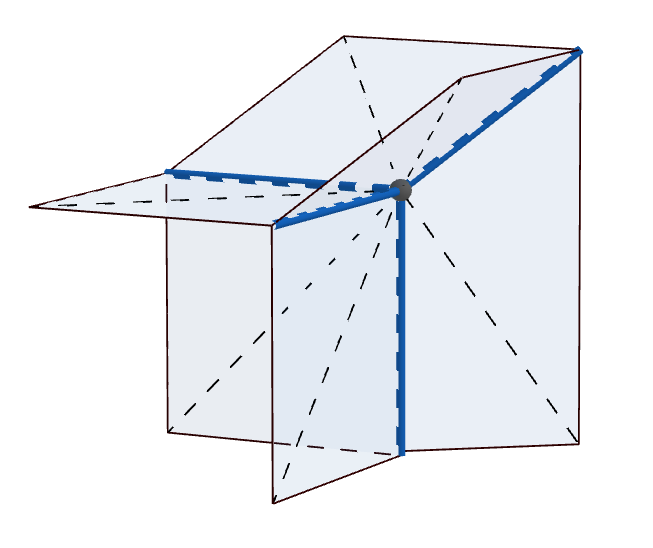}
\caption{The projective fan of the matroid $U_{3,4}$ drawn in $\R^4/(1,\dots, 1)$ as described in Example \ref{ex:uniform}.}
\label{fig:fanU34}
\end{figure}

\begin{example}\label{ex:uniform}
The uniform matroid of rank $k+1$ on $n$ elements will be denoted $U_{k+1,n}$.
The flats of the  matroid  $M = U_{k+1,n}$ are all subsets of $\{1, \dots, n\}$ of size less than or equal to $k$ and $\{1, \dots, n\}$. 
Therefore, the faces of top dimension of $\Sigma_M$ are in bijection with ordered subsets of size $k$. For example, the ordered set $\{i_1 <i_2 < \dots < i_k\}$ 
corresponds to a chain of flats $$\F = \{\emptyset \subsetneq  F_1 \subsetneq  \dots \subsetneq  F_k\subsetneq  E\}$$ where $F_1 = \{ i_1\} $ and 
$F_{j+1} = F_{j} \cup \{i_{j+1}\}$, 
which  in turn corresponds to the cone $$\sigma_{\mathcal{F}} = \langle v_{F_1}, \dots ,  v_{F_k} ,v_{E}, -v_E \rangle_{\geq 0 }.$$
The projective fan for $U_{3, 4}$ is shown in Figure \ref{fig:fanU34}. The dotted rays in this fan correspond to the $6$ rank $2$ flats.

There  are  coarser fan structures on the support of the fan $\Sigma_{M}$. For example, the fan with maximal cones $\sigma_I =  \langle v_{i_1} , \dots, v_{i_k},  v_{E}, -v_E \rangle$ where $I = \{i_1, \dots, i_k\}$ ranges over all subsets $I$ of $E$ of size $k$,  has the same support as $\Sigma_M$. This in known as the coarse matroid  fan of $M$ \cite{ArdilaKlivans}. 

\end{example}

\subsection{Necklace arrangements}\label{sec:necklace}

We now give an equivalent reformulation of Condition (\ref{cond2}) in Definition \ref{def:realstructure} in the case of matroid fans. 
Notice first that there is a  one  to one correspondence between real phase structures
on $\Sigma_M$ and $\projFan_M$ induced by the projection
$\Z_2^E \to \Z_2^E / (1,\dots,1)$.

Suppose $\EEE$ is a real phase structure on a $d$-dimensional rational polyhedral fan $\Sigma$ and let $\tau$  be a codimension $1$ face of $\Sigma$. Then the affine subspaces $\EEE(\sigma_i)$ where $\sigma_i$ are the facets adjacent to $\tau$ all contain the direction of the $(d-1)$-dimensional
linear space $T_{\Z_2}(\tau)$. The even covering property at the codimension one face $\tau$ can be equivalently checked on the lines in $\Z_2^n / T_{\Z_2}(\tau)$ obtained as projections of the $\EEE(\sigma_i)$'s.

Given an arrangement of lines $L_1, \dots, L_k \in  \Aff_1(V)$, its \emph{intersection complex} is the simplicial 
complex that consists of a vertex for every line and a simplex on the vertices $i_1, \dots, i_q$ for every 
point in $L_{i_1} \cap \dots \cap L_{i_q}$. 

\begin{definition} \label{def:necklace}
An arrangement of lines $L_1, \dots, L_k \in  \Aff_1(V)$ is a \emph{necklace of lines}  if its intersection complex is a cycle graph.
An arrangement of subspaces $E_1, \dots , E_k \in \Aff_d(V)$ whose tangent spaces 
share a $d-1$-dimensional linear space $W$ is called a 
\emph{necklace arrangement} if the projection to $V/W$ yields lines $L_1, \dots, L_k$ forming a necklace of lines. 
\end{definition}

\begin{remark} \label{remspecialcase}
  We will use this definition exclusively for vector spaces over $\Z_2$. 
	Under this assumption, two lines $L_1, L_2$ form a necklace if and only if $L_1 = L_2$.
	If a  necklace arrangement consists of more than two lines, then these lines must be  pairwise distinct.
	
	For subspace arrangements of higher dimension, note  that the definition of necklace arrangement is independent 
	of the choice of $W$. Indeed, if this choice is not unique, then 
	the affine spaces are all parallel and hence form a necklace if and only if 
	$k=2$ and $E_1 = E_2$. 
\end{remark}

\begin{figure}
\includegraphics[scale=0.3]{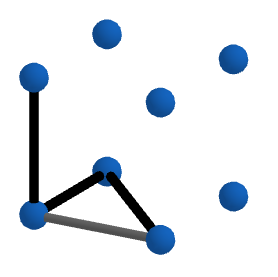} 
\hspace{0.5cm}
\includegraphics[scale=0.3]{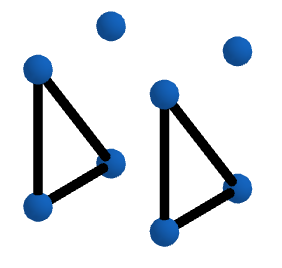} 
\hspace{0.5cm}
\includegraphics[scale=0.3]{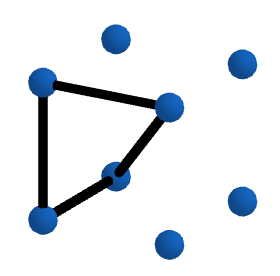} 
\caption{Three line arrangements in $\Z_2^3$ from Example \ref{ex:necklaceArrangements}.}
\label{fig:necklaceArrangements}
\end{figure}

\begin{example}\label{ex:necklaceArrangements}
Figure \ref{fig:necklaceArrangements} shows three different line arrangements in $\Z_2^3$. The points in $\Z_2^3$ are represented as vertices of a cube. Lines in $\Z_2^n$ are in correspondence with pairs of  points. So in the figure a line is represented by an edge joining two points. The first line arrangement consists of $4$ lines, and the $4$ lines considered as subsets of $\Z_2^3$ do not form an even cover in the sense of Definition \ref{def:evencovering}. In the second example, there are $6$ lines, which form an even cover but do not form a necklace arrangement. The third example is collection of $4$ lines forming  a necklace arrangement. 
\end{example}

We  can establish the following alternative for condition (2) in  Definition \ref{def:realstructure} in the case of matroid fans. 

\begin{enumerate}
	\item[(2')] For every codimension one face $\tau$ of $\Sigma$ with facets $\sigma_1, \dots , \sigma_k$ adjacent to it, the subspaces 
	$\E(\sigma_1), \dots, \E(\sigma_k)$ form a necklace arrangement.
\end{enumerate}

\begin{lemma} \label{lem:2'equiv}
  Let $ \mathcal{E} \colon \text{Facets}(\Sigma) \to \Aff_d(\Z_2^n)$ be a map satisfying
	condition (1) from Definition \ref{def:realstructure}.
	If $\E$ satisfies condition (2'), then it also satisfies condition (2). 
	Moreover, if $\Sigma = \Sigma_M$ (or $\Sigma = \projFan_M$) is an affine (or projective) matroid fan, then the two conditions are equivalent.
\end{lemma}

\begin{proof}
Suppose  Condition (2') is satisfied, then at each codimension one face $\tau $ of $\Sigma$ the intersection complex of the subspaces $\E(\sigma_1), \dots, \E(\sigma_k)$ is a necklace arrangement. 
In particular, every element in $\E(\sigma_1) \cup \dots \cup \E(\sigma_k)$  is contained in exactly two of the affine spaces. This implies the first statement. 

	Let us now assume $\Sigma = \Sigma_M$ for some matroid $M$. The statement for projective fans is equivalent. 
	Assume that $\EEE$ satisfies condition (2).   
	For a face $\tau$ of codimension $1$ of $\Sigma_M$ and a facet $\sigma_i$ adjacent to $\tau$, let 
	$v_i \in \Z^n/ T_{\Z}(\tau)$ be a non-zero  
	integer vector generating the image of $T_{\Z}(\sigma_i)$ under the quotient by $T_{\Z}(\tau)$.  Let $\overline{v}_i \in \Z_2^n/ T_{\Z_2}(\tau)$ denote the mod $2$ reduction 
	of $v_i$. 
	Then the directions of the  lines in the necklace line arrangement in    $\Z_2^n / T_{\Z_2}(\tau)$ are $\overline{v}_i$.
	Any cycle in the intersection complex corresponds to a non-trivial linear relation among the vectors $\overline{v}_i$  
        in $\Z_2^n / T_{\Z_2}(\tau)$. 
	In the case of matroid fans, there is a unique, up to scalar,  linear relation among the $\overline{v}_i$'s, namely their sum is zero.
	To see this, note that the chain of flats $\FFF = \{\emptyset \subsetneq F_1\subsetneq \dots \subsetneq F_k \subsetneq E\} $ associated to $\tau$ contains exactly one gap. Namely, we have $\rk(F_s) = \rk(F_{s-1}) +2$ for a unique $s$ and $\rk(F_j) = \rk(F_{j-1}) +1$, otherwise. 
	Denote by $H_1, \dots , H_k$ all the flats such that $F_{s-1} \subsetneq H_i \subsetneq F_{s}$, $1 \leq i \leq k$. 
	By the covering axiom for flats, the sets $H_i \setminus F_{s-1}$ form a partition of $F_{s} \setminus F_{s-1}$.
	Hence 
	\[
	  \sum_{i=1}^l v_{H_i} = v_{F_s} - (k-1) v_{F_{s-1}}
	\]
	as in \cite[Remark 4.4 b)]{Rau-TropicalLefschetzHopf}, which implies the claim on the $\overline{v}_i$'s.
	%
	%
\end{proof}

\begin{definition}\label{def:necklaceOrder} 
Given a finite set $S$, a \emph{necklace ordering} of $S$
is an equivalence class of two cyclic orderings of $S$, which are related by reversing the order.
For example, a cycle graph defines a necklace ordering of its vertices. 
\end{definition}

\begin{remark}\label{rem:necklaceordering}
A real phase structure on a matroid fan $\Sigma_M$ determines at every codimension one face  of $\Sigma_M$ a necklace ordering of the facets of $\Sigma_M$ adjacent to the codimension one face.
The necklace ordering is defined by the cycle graph of the necklace line arrangement at each codimension one face.  A reorientation of a real phase structure in the sense of Definition \ref{def:reorientation} induces the same necklace ordering of facets adjacent to codimension one faces as the original real phase structure. 
\end{remark}

\begin{ques}\label{ques:necklace}
Is a real phase structure on a matroid fan determined, up to reorientation, by the induced necklace ordering 
of the facets at each codimension one face of the fan?
\end{ques}

In general, we may ask for a description of the set of real phase 
structures which produce a fixed collection of necklace orderings at codimension one faces.
Using the correspondence between real phase structures on matroid fans and orientations of matroids which we will prove here, we can translate the question to one about oriented matroids:
Up to reorientations, is an oriented matroid determined by its rank $2$ minors?

\begin{example}\label{ex:line}
Consider the projective fan of the uniform matroid $M = U_{2, n}$. The fan $\projFan_M \subset \R^{n-1}$  has $n$ edges generated by the images of the vectors $v_1 = -e_1, \dots, v_{n} = -e_{n}$ in $\R^{n} / \langle (1,\cdots,1) \rangle$. Denote by $\rho_i$ the image of the vector $v_i$. Note that $\sum \rho_i=0$. 
Choosing a real phase structure on $\projFan_M$ amounts to choosing the following ingredients:
\begin{enumerate}
\item A necklace  ordering of the $n$ edges corresponding to $ \rho_{i_1}, \dots, \rho_{i_{n}}$;  
\item A point $p \in \Z_2^{n-1}\simeq \Z_2^{n} / \langle (1,\cdots,1) \rangle$ that serves as the  intersection point  of $\E(\rho_{i_1})$ and $\E(\rho_{i_2})$. 
\end{enumerate} 
From this information, a collection of affine lines $\E(\rho_{i})$ satisfying the conditions of Definition \ref{def:realstructure} can be uniquely recovered. For example, the choice of point $p$ determines both $\EEE(\rho_{i_1})$ and $\EEE(\rho_{i_2})$, since their tangent spaces are fixed. By the necklace arrangement property, the point $p + \sum_{k = 2}^{j-1} v_{i_k}$ is in the affine line  $\EEE(\rho_{i_j})$ for $j \geq 3$, where the vector sum is considered mod $2$. 
This determines all of  the affine lines $\EEE(\rho_{i_j})$.

Figure \ref{fig:reallineU34} shows the fan $\projFan_M \subset \R^4 / \langle (1, 1, 1, 1) \rangle \cong \R^3$ for $M = U_{2,4}$ together with an assignment of affine spaces along its edges that  determine a real phase structure. The induced necklace ordering of the facets  is $\sigma_2,   \sigma_3, \sigma_1, \sigma_4$.
From Figure \ref{fig:reallineU34}, we see that the point $p$ is contained in the intersection $\EEE(\sigma_2) \cap \EEE(\sigma_3)$. 
 If we set $p = (0, 0, 0, 0) \in  \Z_2^4 / \langle (1, 1, 1, 1) \rangle$, then the corresponding necklace of lines is the last of the three arrangements in  $\Z_2^3 \cong \Z_2^4 / \langle (1, 1, 1, 1) \rangle$ depicted in Figure \ref{fig:necklaceArrangements}. 
\end{example}

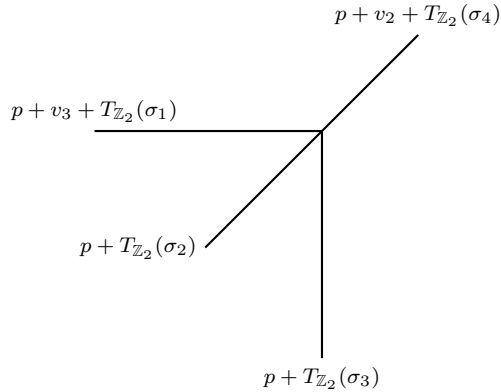
\begin{figure}
\begin{tikzpicture}
\coordinate (O) at (0,0,0);
    \draw[thick, -] (O) --++ (-3,0,0)   node[above] {$  p + v_3 + T_{\Z_2} (\sigma_1)$ };
    \draw[thick, -] (O) --++ (0,-3,0) node[below] {$p  + T_{\Z_2} (\sigma_3) $};
    \draw[thick, -] (O) --++ (0,0,4)  node[below, left] {$ p  + T_{\Z_2} (\sigma_2)$ };
    \draw[thick, -] (O) --++ (0.5,0.5,-2) node[above] {$ p  + v_2 + T_{\Z_2} (\sigma_4) $};
\end{tikzpicture}
\caption{The projective fan of the matroid $U_{2, 4}$ from Example \ref{ex:line} drawn in $\R^4/\langle (1, \dots, 1) \rangle$ with the labelling of the faces indicating the assignment of an affine space $\EEE(\sigma_i)$ parallel to 
$T_{\Z_2}(\sigma_i)$. 
}\label{fig:reallineU34}
\end{figure}

\begin{example}\label{ex:hyperplane}
For the matroid $M = U_{n-1, n}$  we will show that there is a unique real phase structure on $\projFan_M$ up to reorientation. Such
real phase structures were considered in \cite{RS}. 
  
 By \cite[Lemma 3.14]{RS}, a real phase structure $\EEE$ on $\projFan_M$ 
  satisfies $$| \bigcup_{\sigma} \EEE(\sigma) | = 2^{n-1} - 1.$$
    Therefore there is exactly one element $\varepsilon$ in the 
    complement $\Z_2^{n-1} \backslash \cup_{\sigma} \EEE(\sigma)$. Up to reorientation we can suppose that $\varepsilon = (0, \dots, 0)$.
   
   Since  $\projFan_M$ is of codimension one in $\R^{n-1}$, for each facet $\sigma$ of $\projFan_M$  there is a choice of exactly two affine subspaces of $\Z_2^{n-1}$ which are parallel to the reduction of the span of $\sigma$ in $\Z_2^{n-1}$.   One of these spaces is an honest vector subspace and hence contains $(0, \dots, 0)$. Therefore, if we are to associate to each $\sigma$ an affine subspace $\E(\sigma)$ and wish to avoid that it contains $(0, \dots, 0)$, then the choice of affine space at each top dimensional face is determined.  This demonstrates that there is at most one real phase structure on $\projFan_M$,  up to reorientation in the sense of Definition \ref{def:reorientation}.
\end{example}

\subsection{Deletion and contraction of real phase structures} \label{subsec:deleteandcontract}

We briefly recall the notion of minors of a matroid. Let $M$ be a matroid with ground set $E$ and rank function $\rk$. For $S \subset E$, then the 
\emph{deletion} of $S$ is the matroid 
$M\backslash S$ with ground set   $E \setminus S$ and rank function $\rk_{M\backslash S}(A)= \rk_{M}(A)$. The \emph{contraction} of $M$ by $S$ is the matroid  $M/S$ whose ground set is again $E\setminus S$ and  rank function $\rk_{M/S}(A)= \rk_M(A\cup S)-\rk_M(A)$. 
Lastly, the  \emph{restriction} of $M$ to $S$ is the matroid $M|_S$ whose ground set is $S$ and rank function $\rk_{M|_S}$ is the restriction of $\rk_M$. Notice that $M|_S = M \backslash S^c$, where $S^c = E \setminus S$. 
A \emph{minor} of a matroid $M$ is any matroid obtained from $M$ by a sequence of deletions and contractions. A deletion or contraction is called \emph{elementary} when the set $A$ is a singleton.

For any subset $A \subset E$, we denote by $p_A \colon \R^E \to \R^{E \setminus A}$ the projection 
which forgets the coordinates $x_i$ for all $ i \in A$. If the matroid $M$ has loops $L \subseteq E$, we use the same notation
for the projection $p_A \colon \R^{E\setminus L} \to \R^{E \setminus (L\cup A)}$. 
We also use the shorthand $p_i$ in the case $A = \{i\}$. 

If $i$ is a loop or coloop of $M$, then $M\backslash i = M/i$, so deletion and contraction are equivalent. The  support of $\Sigma_{M \backslash i}$ is the image of the projection of  the matroid fan $\Sigma_M$ under the projection $p_i$.
Note that this is also true if $i$ is a loop, in which case, according to our conventions, $\Sigma_{M \backslash i} = \Sigma_M$ and $p_i = \text{id}$. 
If $i$ is not a coloop, then the facets  of $\Sigma_{M \backslash i} $ are the projections of facets of $\Sigma_M$ whose dimensions are preserved under $p_i$.

Suppose that $i$ is not a loop of $M$.
Note that by our convention regarding loops, we have $\Sigma_{M / i} = \Sigma_{M/ \cl(i)}$. 
The support of the matroid fan of  $M/ i$ is the set
$\{ x \in \R^{E \setminus \cl(i)} \ | \ |p^{-1}_{\cl(i)}(x) | > 1 \}$.
The facets of $\Sigma_{M /i}$ are the images of facets of $\Sigma_M$ whose dimensions are \emph{not} preserved under the projection by $p_{\cl(i)}$. 
More details on the geometry of $\Sigma_M$,  $\Sigma_{M\backslash i}$, and $\Sigma_{M/i}$ and their relations under $p_i$ 
can be found in \cite[Section 2]{ShawIntMat} and also \cite[Section 3]{FR-DiagonalTropicalMatroid}.

A real phase structure on the fan of a matroid  $M$ induces canonical real phase structures on the fans of all  minors of $M$. 
We will describe the induced real phase structures for elementary deletions and contractions of a  matroid $M$.
The geometric idea is very simple: given a facet $\sigma$ of the matroid fan of a minor, 
we pick a facet $\tilde{\sigma}$ of $\Sigma_M$ that projects to $\sigma$. Then the affine space associated to 
$\sigma$ is the projection of $\EEE(\tilde{\sigma})$.
Given that we work with the fine subdivision of $\Sigma_M$ induced by the lattice of flats, the choice of $\tilde{\sigma}$ 
is in general not unique. To simplify the proofs in the following sections, we make a specific choice for $\tilde{\sigma}$.
However, Definition \ref{InducedRealStructures} is  independent of this choice,  as discussed after the definition.
For $\sigma$ a facet of $\Sigma_{M \backslash i}$, let $p^*_i(\sigma)$ be the facet of $\Sigma_M$ which is obtained by taking the closure in $M$ of all the flats of $M \backslash i$  occurring
in the chain of flats describing $\sigma$ when $i$ is not a coloop of $M$. 
If $i$ is a coloop of $M$, prolongate the chain by one piece by adding $i$ everywhere. 
Note that $p_i(p^*_i(\sigma)) = \sigma$. Moreover, if $i$ is not a 
coloop then  $p^*_i(\sigma)$ is the unique facet that projects to $\sigma$. For $\sigma $ a facet in $\Sigma_{M/ i}$ corresponding to  the chain of flats 
\[\emptyset = \subsetneq F_1 \subsetneq \dots \subsetneq F_k \subsetneq E \setminus \cl(i),\]
set  $p_{\cl(i)}^{\diamond} (\sigma)$ to be the facet
of $\affFan_M$ given by the chain 
\[\cl(\emptyset) \subseteq \cl(i) \subsetneq F_1 \sqcup \cl(i) \subsetneq \dots \subsetneq F_l \sqcup \cl(i) \subsetneq E.\]
Note that $p_{\cl(i)}(p_{\cl(i)}^{\diamond} (\sigma)) = \sigma$. By abuse of notation, we use the same letter $p_A$ 
for the reduction mod $2$ counterpart $p_A \colon \Z_2^E \to \Z_2^{E \setminus A}$.

\begin{definition} \label{InducedRealStructures}
	Let $\EEE$ be a real phase structure for the matroid fan $\Sigma_M$ and choose $i \in E$. 
	The \emph{deletion} $\EEE \backslash i$
	 is the real phase structure on  $\Sigma_{M \backslash i}$ 
	given by 
	\[
	  (\EEE \backslash i) (\sigma) = p_i (\mathcal{E} (p_i^* (\sigma)))
	\]
	for any facet $\sigma$ of $\Sigma_{M \backslash i}$. 
	The \emph{contraction} $\EEE / i$ is the real phase structure on  $\Sigma_{M / i}$ 
	given by 
	\[
	  (\EEE / i) (\sigma) = p_{\cl(i)} (\mathcal{E} (p_{\cl(i)}^{\diamond} (\sigma)))
	\]
	for any facet $\sigma$ of $\Sigma_{M / i}$.
\end{definition} 

As previously mentioned, for $\sigma$ a facet of either $\Sigma_{M\backslash i} $ or $\Sigma_{M/ i}$, we are free to  replace $p_i^* (\sigma)$ or $p_{\cl(i)}^{\diamond} (\sigma)$ in the above definition with any facet of $\Sigma_M$ which 
projects onto $\sigma$ under $p_i$ or $p_{\cl(i)}$, respectively. Any such facet will be contained in the same facet of the coarsest subdivision of the support of  $\Sigma_M$ as the facets  $p_i^* (\sigma)$ or $p_{\cl(i)}^{\diamond} (\sigma)$, respectively. Therefore,  Conditions (1) and (2) of a real phase structure imply that $\E$ must assign the same affine space to any such choice of face. 

\begin{proposition}\label{realstructure_contraction_deletion}
The maps $\E \backslash i$ and $\E/ i$ from  	Definition \ref{InducedRealStructures} define  real phase structures on 
$\Sigma_{M \backslash i }$ and $\Sigma_{M / i}$ respectively.
\end{proposition}

\begin{proof}
  We must show that the maps $\E \backslash i$ and $\E /i$ satisfy Conditions (1) and (2) of Definition \ref{def:realstructure}.
	We set $p = p_i$ or $p = p_{\cl(i)}$ depending on whether we consider the deletion or the contraction. 
	Let $\sigma$ be a facet of either $\Sigma_{M\backslash i} $ or $\Sigma_{M/i}$. 
	Let  $\tilde{\sigma} = p^{*} (\sigma)$ or $\tilde{\sigma} = p^\diamond(\sigma)$. 
	The projection of an affine space along $p$ remains an affine space. 
	Moreover, since we have $\sigma = p(\tilde{\sigma})$ and hence also
	 $T( \sigma) = p (T ( \tilde\sigma ) )$, Condition (1) holds for both $\E \backslash i$ and $\E / i$.
	 
	Let us now check Condition (2). Let $\tau$ be a codimension one face of 
	either $\Sigma_{M\backslash i} $ or $\Sigma_{M/i}$
	and set $\tilde{\tau} = p^{*} (\tau)$ or $\tilde{\tau} = p^\diamond(\tau)$, respectively. 
	If $\dim \tilde{\tau} = \dim \tau +1$, then the arrangements of
	affine subspaces around $\tau$ and $\tilde{\tau}$ agree after quotienting by
	$T_{\Z_2}(\tau)$ and $T_{\Z_2}(\tilde{\tau})$, respectively. 
	
Now let us assume that $\dim \tilde{\tau} = \dim \tau$. 
	Let $\tilde{\sigma}_1, \dots, \tilde{\sigma}_k$ denote the facets of $\affFan_M$ adjacent to $\tilde{\tau}$.
	Let $J  = \{ j \ |  \ 1 \leq j \leq k , \dim \tilde{\sigma}_{j} = \dim p(\tilde{\sigma}_{j}) \}$.
	Then the arrangement of affine subspaces around $\tau$ in the induced real phase structure consists of the affine spaces 
	     $p(\EEE(\tilde{\sigma}_{j}))$ for $j \in J$. 
	Given a point $\varepsilon \in  p(\EEE(\tilde{\sigma}_{j_0}))$ for some $j_0 \in J$, first notice that	
		\begin{align*} 
		|\{j \in J \colon  \varepsilon \in p(\EEE(\tilde{\sigma}_{j} ))\} |
		&   = |\{ (\tilde{\varepsilon}, j) \ | \ \tilde{\varepsilon} \in p^{-1}(\varepsilon) \cap \EEE(\tilde{\sigma}_j) \text{ and } j\in J\} |  \\
			& \equiv    |\{ (\tilde{\varepsilon}, i) \ | \  \tilde{\varepsilon} \in p^{-1}(\varepsilon) \cap \EEE(\tilde{\sigma}_i), 1 \leq i \leq k\}| \mod 2.
	\end{align*}
	The equality and congruence follow from the fact that 
	$|\{\EEE(\tilde{\sigma}_i) \cap p^{-1}(\varepsilon)\}|= 1$ if $\dim \tilde{\sigma}_i = \dim p(\tilde{\sigma}_i)$ and $|\{\EEE(\tilde{\sigma}_i) \cap p^{-1}(\varepsilon)\}|$ is  even otherwise. 
	The last expression is a sum of even numbers by the fact that $\E$ is a real phase structure on $\Sigma_M$. Hence, any point $\varepsilon \in  p(\EEE(\tilde{\sigma}_{j_0}))$ for any $j_0 \in J$ 
	is covered an even number of  times by the affine spaces around $\tau$ which proves condition (2). 
\end{proof}

  Note that we can iterate the operations of deletion and contraction to construct general \emph{minors}
	$\EEE \backslash A / B$. To justify the notation, we need to show that the result is 
	invariant under reordering the sequence of deletions and contractions.

	\begin{proposition} 
			Let $\EEE$ be a real phase structure for the matroid fan $\Sigma_M$ of the matroid $M$ and choose $i \neq j  \in E$. 
			Then 
			\begin{align*} 
				\EEE \backslash i \backslash j &= \EEE \backslash j \backslash i, \\
				\EEE \backslash i / j &= \EEE / j \backslash i, \\
				\EEE / i / j &= \EEE / j / i.
			\end{align*}
	\end{proposition}

	\begin{proof}
		We start with the first two equalities. 
		Note that by definition $\EEE \backslash i = \EEE / i$ if $i$ is a coloop. 
		Hence we may assume that not both $i$ and $j$ are coloops in the first equation
		and $i$ not a coloop in the second equation, the exceptions being covered 
		by the third equation. 
		Under these assumptions, the first two equations hold since $p_i^* \circ p_j^* = p_j^* \circ p_i^*$
		and $p_i^* \circ p_j^\diamond = p_j^\diamond \circ p_i^*$.	
		
		For the last equality, 
	notice that for any $i, j$ the projections $p_{\cl(i)}$ and $p_{\cl(j)}$ commute and the composition is $p_{\cl(i,j)}$.
		We are asked to compare the projections under $p_{\cl(i,j)}$ of  the affine spaces $\E(\sigma_1), \E(\sigma_2)$ 
		where $\sigma_1, \sigma_2$ are two faces of $\affFan_M$ 
		associated to the chains of flats
		\begin{align} 
			\cl(\emptyset) \subseteq \cl(i) \subseteq \cl(i,j) \subsetneq F_1 \subsetneq \dots \subsetneq F_k \subsetneq E, \\
			\cl(\emptyset) \subseteq \cl(j) \subseteq \cl(i,j) \subsetneq F_1 \subsetneq \dots \subsetneq F_k \subsetneq E.
		\end{align}
		We may assume $\rk(i,j) = 2$, since otherwise $\sigma_1 = \sigma_2$. Then $\sigma_1$ and $\sigma_2$ are adjacent to
		the codimension one face $\tau$ associated to 
		\[
			\cl(\emptyset) \subsetneq \cl(i,j) \subsetneq F_1 \subsetneq \dots \subsetneq F_k \subsetneq E.
		\]
		Note that $p_{\cl(i,j)}(T(\sigma)) = p_{\cl(i,j)}(T(\tau))$ for every facet $\sigma$ adjacent to $\tau$. 
		In particular, if two affine subspaces in the necklace arrangement intersect, their projections under 
		$p_{\cl(i,j)}$ agree. 
		Therefore the necklace condition (2') implies that $p_{\cl(i,j)}(\EEE(\sigma))$ is the same affine space for every $\sigma$ adjacent to $\tau$.  
		In particular, we have $p_{\cl(i,j)}(\EEE(\sigma_1)) = p_{\cl(i,j)}(\EEE(\sigma_2))$. This  proves the claim.	
	\end{proof}

\section{Matroid orientations and real phase structures} \label{sec:Orientedmatrealphase} 

\subsection{From oriented matroids to real phase structures}\label{subsec: orientedmatroids}
Here we will produce a real phase structure on a matroid fan from an oriented matroid.
We will use the covector description of oriented matroids. 
For an oriented matroid  $\M$, on ground set $E$, the covectors of $\M$ are 
a subset $\CC \subseteq  \{0, + 1 , -1 \}^{E}$.  Let $X\in\CC$. For $i\in E$, the $i$-th coordinate of $X$ is denoted by $X_i$. The positive and negative parts of $X$ are respectively 
$$
X^{+}:=\left\lbrace i\in E \mid X_i=+1 \right\rbrace,
$$
and 
$$
X^{-}:=\left\lbrace i\in E \mid X_i=-1 \right\rbrace.
$$
 The support of $X$ is 
 $$
 \text{Supp}(X):=\left\lbrace i\in E \mid X_i\neq 0\right\rbrace.
 $$
The composition operation $\circ$ on covectors $X$ and $Y$ is defined by 
$$(X \circ Y)_i = 
\begin{cases} 
X_i \text{ \ \  if } X_i \neq 0\\
 Y_i \text{ \ \ \  if } X_i = 0 .
 \end{cases}
$$
 The separation set $S(X,Y)$ is defined by 
 $$
 S(X,Y):=\left\lbrace i\in E \mid X_i=-Y_i \neq 0 \right\rbrace.
 $$
The covectors of an oriented matroid  satisfy the following axioms: 
\begin{enumerate}
\item  $0 \in \CC$
\item $X \in \CC$ if and only if $-X \in \CC$
\item $X, Y \in \CC$ implies that $X \circ Y \in \CC$
\item If $X, Y \in \CC$ and $i \in S(X, Y) $ then there exists a $Z \in \CC$ such that $Z_i = 0$ 
and $Z_j = (X \circ Y)_j = (Y \circ X)_j$ for all $j \not \in S(X, Y)$. 
\end{enumerate}

The set of covectors $\CC$ forms a lattice under the partial order  $0 < +1, -1$ considered coordinatewise. 
There is a forgetful map  $\phi$ from oriented matroids to matroids which preserves rank and the size of the ground set. 
Given an oriented matroid $\M$, we let $\underline{ \M} = \phi(\M) $ denote its underlying matroid. 
We can describe the forgetful map on the level of the covector lattice $\CC$ of $\M$ and the lattice of flats  $\L$ of $M$. 
Given a covector $X \in \mathcal{C}$ the forgetful map assigns 
$\phi(X) = \text{Supp}(X)^c \in \L$ where $A^c$ denotes $E \setminus A$. 
The image of a covector of the oriented matroid under the forgetful map is a flat of the underlying matroid \cite[Proposition 4.1.13]{bjorner}. 

The set of \emph{topes} $\mathcal{T}$ are the maximal covectors with respect to the partial order on $\mathcal{C}$. If the underlying matroid of 
$\M$ has no loops we have 
$\mathcal{T} \subseteq \{+1, -1\}^{|E|}$.
Let $\M$ be an oriented matroid with collection of topes $\TTT$ and underlying lattice of flats $\LLL$. For  $F \in \LLL$ and $T \in \TTT$, we denote by $T \setminus F \in \{0,+1,-1\}^E$ the vector 
obtained by setting all coordinates
in $F$ to $0$. We say $F$ is \emph{adjacent to} $T$ if $T \setminus F \in \CCC$.
More generally, given a flag $\mathcal{F} := F_0 \subsetneq F_1 \subsetneq  F_2 \subsetneq \dots \subsetneq F_k$ of flats in $\mathcal{L}$,
we define the set of topes adjacent to $\FFF$ by 
$$\mathcal{T}(\mathcal{F}) = \{ T \in \mathcal{T} \ | \ T \setminus F_i \in \CCC \text{ for all } i = 0,\dots,k \}.$$

\begin{example} \label{RealHyperplanes}
 A set $H_1,\dots, H_n$ of hyperplanes in $\R^r$ defined by linear forms $l_1,\dots,l_n$ produces an oriented matroid on $\left\lbrace 1,\dots, n \right\rbrace$. The underlying matroid on $\left\lbrace 1,\dots, n \right\rbrace $ is given by the rank function $\rk(A)=\text{codim} (\cap_{i\in A} H_i)$.  
A covector corresponds to a cell of the decomposition of $\R^r$ induced by the positive regions $H_i^+=\left\lbrace l_i(x)\geq 0 \right\rbrace$ and negative regions $H_i^-=\left\lbrace l_i(x)\leq 0 \right\rbrace$. Assuming that none of the linear forms are identically equal to zero, the topes are in bijection with the cells in the complement of the arrangement.
The flat associated to a covector is in bijection with the set of hyperplanes containing the corresponding cell.
Figure \ref{fig:projectivearrangement}, shows the intersection of an arrangement of four planes in $\R^3$ with a sphere.  The underlying matroid of this arrangement is  the uniform matroid $U_{3,4}$. 
 There are $14$ cells of dimension two in the subdivision of the sphere induced by the intersections of the four planes. These are the topes of the oriented matroid. Each cell of the complement is labelled by a tuple $\{+, -\}^4$ corresponding to the sign of the linear forms $l_1, \dots, l_4$ evaluated at a point in the open cell.  
\begin{figure}
\centering
\includegraphics[scale=0.25]{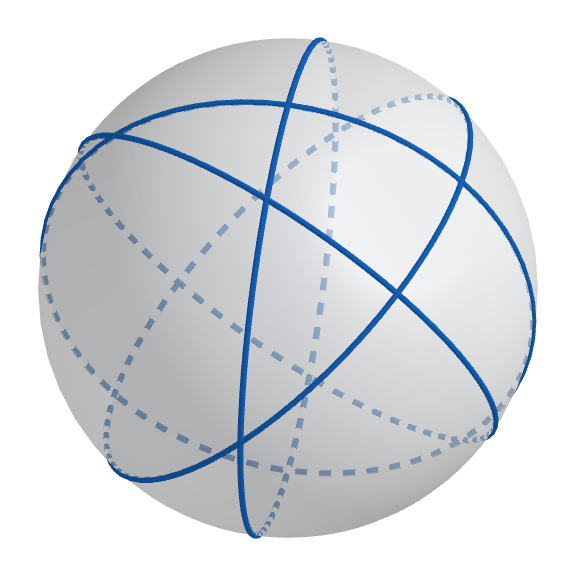}
\put(-95, 50){$l_1$}
\put(-60, 90){$l_2$}
\put(-55, 50){$l_3$}
\put(-95, 125){$l_4$}
\caption{The intersection of a real arrangement of $4$ generic planes in $\R^3$ with the unit sphere. Assigning the covector $(+, +, +, +)$ to the region bounded by the spherical triangle  facing the viewer formed by $l_1, l_2, l_3$ determines the covectors of all regions.  }
\label{fig:projectivearrangement}
\end{figure}
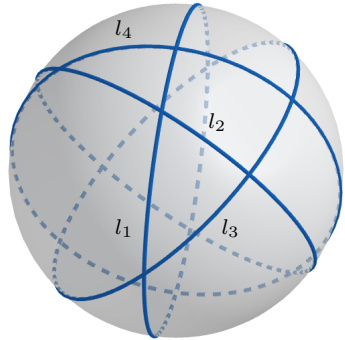
\end{example}

\begin{lemma}\label{lemma:numberoftopes}
Let $\M$ be a loopfree oriented matroid of rank $d$ and set $M = \underline{ \M}$.  
For any flag $\F = \{\cl(\emptyset) = F_0 \subsetneq F_{1} \subsetneq \dots \subsetneq F_{k} \subsetneq F_{k+1} = E\}$ of $M$ we have 
$$|{\mathcal{T}}({\F})| =  |\chi_{M_{\F}}(-1)|,$$
where $\chi_{M_{\F}}(t)$ denotes the characteristic polynomial of the matroid $$M_{\F} = \bigoplus_{i = 1}^{k+1} M|F_i/F_{i-1}.$$
In particular, if $\mathcal{F}$ is a maximal flag, then $|{\mathcal{T}}({\F})| = 2^{d}$.
\end{lemma}

\begin{proof}
The number of topes of a loopfree oriented matroid $\mathcal{N}$  is equal to  $| \chi_{\underline{\mathcal{N}}}(-1)|$ \cite{Zaslavsky}.
Moreover, 
by \cite[Proposition 3.7.11]{bjorner} the set ${\mathcal{T}}({\F})$ is equal to 
the set of topes of the oriented matroid $\MMM_{\F}= \bigoplus_{i = 1}^{k+1} \MMM|F_i/F_{i-1}$.
Since $\underline{\MMM}_\FFF = M_\FFF$, the statement follows. 
 \end{proof}

\begin{remark}
We would like to make the following remark on our choice of conventions. In this paper, we use both the multiplicative and additive notation on the group of two elements $( \{0, 1\}, + ) $ and $(\{1, -1\}, \cdot)$. 
When speaking of real phase structures we work with vector spaces over  $\Z_2$, therefore it is preferable to use the additive notation and denote the field of two elements 
by $\{0, 1\}$. On the other hand it  is tradition that the covectors of oriented matroids take values in $\{0, +, -\}$ and we also make use of the group structure on $\{+, -\}$. 
We routinely use the notation $\varepsilon$ to denote elements of the field $\{0,1\} $ or of vector spaces over this field.  We use uppercase roman letters, for example, $X, Y, T$, to denote 
covectors. 
Covectors can be multiplied entry by entry and
this operation is denoted by $T\cdot T'$. 

To go from the additive group notation to the multiplicative group notation for vectors, we use $(-1)^{\varepsilon} = ((-1)^{\varepsilon_1} , \dots, (-1)^{\varepsilon_n}) $, where $\varepsilon = (\varepsilon_1, \dots, \varepsilon_n)$. 
This defines a bijection
\begin{align} \label{eq:ExpMap} 
  \Z_2^E &\to \{+1,-1\}^E, \\
	\varepsilon &\mapsto (-1)^\varepsilon. \nonumber
\end{align}

We avoid going backwards as much as possible 
to avoid writing such perversities as $\log_{-1}$, even though this map makes sense as a discrete logarithm for groups. 
\end{remark}

\begin{definition} \label{def:OrientedToRealPhase}
  Let  $\M$ be a loopfree oriented matroid on the ground set $E$ and  let $\Sigma_M \subseteq \R^E$ the fan of the underlying matroid $M = \underline{\M}$. 
	For every facet $\sigma_\FFF$ of $\Sigma_M$ corresponding to the  maximal flag of flats $\F$, we set
	\[
	  \E_{\M}(\sigma_{\mathcal{F}} ) = \{ \varepsilon \ | \ (-1)^{\varepsilon} \in \mathcal{T}(\F)\} \subseteq  \Z_2^E.
	\]
	If $\MMM$ has loops $L = \cl(\emptyset)$, we set $\EEE_\MMM(\sigma) := \EEE_{\MMM\backslash L}(\sigma) \subseteq \Z_2^{E\setminus L}$.
\end{definition}

In the next two lemmata, we show that $\EEE_\MMM$ defines a real phase structure on $\Sigma_M$. We start with a few useful observations. 
If $T$ is a tope and $X$ any covector then $T \circ X = T$ and $X \circ T$ is always another tope, which is distinct from $T$ if and only if $X^{+}\nsubseteq T^{+}$ or $X^{-}\nsubseteq T^{-}$. 
Given a subset $F \subset E$, 
the \emph{reflection} $r_F(X)$ of a covector $X$ in $F$ 
is given by flipping the signs for all $e \in F$ while keeping
the signs for $e \in E \setminus F$. The reflection of a covector $X$ in a flat $F$ is not always a covector of the oriented matroid.
However, note that if $F$ is adjacent to the tope $T$, then $r_F(T)$ is also a tope.
Indeed, setting $X = T \setminus F$, note that we can rewrite $r_F(T) = X \circ (-T)$,
hence the statement. 

Recall that given any flat $F$ of $M$, there is a vector in $\Sigma_M$ defined by $v_{F}:=\sum_{i\in F}v_i$, where $v_i=-e_i$, see Section \ref{subsecMatroids}. 
We denote by $\varepsilon_F$ the reduction of $v_F$ modulo $2$. 
Note that $ r_F((-1)^\varepsilon) = (-1)^{\varepsilon + \varepsilon_F}$.

\begin{lemma}\label{lemma:topesAffineSpace}
  The set $\EEE_\MMM(\sigma_\FFF)$ from Definition \ref{def:OrientedToRealPhase}
	is a $d$-dimensional  affine subspace parallel to $T_{\Z_2}(\sigma_\FFF)$
	for every facet $\sigma_\FFF$ of $\Sigma_M$. 
\end{lemma}

\begin{proof}
We may assume that $\MMM$ is loopfree. 
Let  $\sigma_\FFF$ be a facet of $\Sigma_M$ with associated maximal flag
	$$\mathcal{F} = \{\emptyset \subsetneq F_1 \subsetneq \dots \subsetneq F_{d-1} \subsetneq F_{d} = E\}.$$
Let $T = (-1)^\varepsilon \in  \mathcal{T}(\mathcal{F})$ be a tope adjacent to $\FFF$.
We need to show that the bijection from (\ref{eq:ExpMap})  maps 
$\varepsilon + \langle \varepsilon_{F_1} , \dots, \varepsilon_{F_{d-1}} , \varepsilon_E  \rangle_{\Z_2}$
to $\TTT(\FFF)$. 
Note that $|\mathcal{T}(\mathcal{F})| = 2^{d}$ by Lemma \ref{lemma:numberoftopes},
so the two sets have the same size. 
Therefore, it suffices to show containment of the image in $\TTT(\FFF)$.
For this, it is enough to prove
that $T' = r_{F_i}(T) \in \TTT(\FFF)$ for every $i = 1, \dots, d$. Note that $T' \setminus F_j = T \setminus F_j$ for $j \geq i$ and $T' \setminus F_j = (T \setminus F_i) \circ (- T \setminus F_j)$ for $j \leq i$.
In both cases, $T' \setminus F_j$ is a covector, and hence $T' \in \mathcal{T}(\mathcal{F})$ as required. 
\end{proof}

\begin{lemma}\label{lemma:orientedMatCircuitArrangement} 
Let $\M$ be a oriented matroid and set $M = \underline{\MMM}$.
Let  $\tau$ be a codimension one face of $\Sigma_M$ and let $\sigma_1, \dots, \sigma_k$ be the adjacent facets. 
Then the subspaces $\EEE_\MMM(\sigma_i)$ for $i =1, \dots,k$ form an even covering. 
\end{lemma}

\begin{proof} 
  Again, we can reduce to the case when $\MMM$ is loopfree. 
  Let $\FFF = \{\emptyset \subsetneq F_1 \subsetneq \dots \subsetneq F_k \subsetneq E\}$ be the chain of flats associated to $\tau$ and pick $T \in \TTT(\FFF)$. 
	Since $\tau$ is of codimension one, there is exactly one rank two jump in $\FFF$. 
  We need to show 
	that there is an even number of completions of $\FFF$ to maximal chains  $\FFF'$ such that $T \in \TTT(\FFF')$ 
	Clearly, such completions are in bijection with the completions of  the chain
	of covectors $(0,\dots, 0) < X_k < X_{k-1} < \dots < X_1 < T$
	with $X_i := T \setminus F_i$.
	Now, recall from \cite[Theorem 4.1.14 (ii))]{bjorner} that the covector lattice of an oriented matroid satisfies 
	the diamond property, that is, all intervals of length $2$ consist of $4$ elements. 
	It follows that there are exactly two such completions, which proves the claim.
\end{proof}

\begin{proposition}\label{prop:oriented2phase}
  The map $\E_{\M} \colon \text{Facets}(\Sigma_M) \to \Aff_{d}(\Z^{E \setminus \cl(\emptyset)})$
	from Definition \ref{def:OrientedToRealPhase}
	defines a real phase structure on $\Sigma_M$.
\end{proposition}

\begin{proof}
By Lemma \ref{lemma:topesAffineSpace} the first axiom of a real phase structure is satisfied. 
 By Lemma \ref{lemma:orientedMatCircuitArrangement} the second axiom is satisfied. 
\end{proof}

To summarise, we have constructed a map from orientations of $M$ (that is, oriented matroids $\MMM$ such that $\underline{\MMM} = M$) 
to real phase structures on $\Sigma_M$,
\begin{equation} \label{eq:orientedtorealstr}
\begin{tikzcd}
\{{\text{Orientations  of }} M \} \arrow[r, "\mathbf{E}"]  & \{{\text{Real phase structures on }} \Sigma_M \},
\end{tikzcd}
\end{equation}
which is given by $\mathbf{E}(\MMM) = \EEE_\MMM$. 
The main result of this paper claims that this map is bijective.
We note that the map is injective, since the topes of $\MMM$ can be recovered from $\EEE_{\MMM}$
as $$\TTT = \{ (-1)^{\varepsilon} \ | \ \varepsilon \in \bigcup_\sigma  \EEE_{\MMM}(\sigma)\},$$ where $\sigma$ runs through all facets of $\Sigma_M$ and moreover, 
an oriented matroid is determined by its collection of topes \cite{topeAxioms}. 

\begin{figure}
\begin{tikzpicture}[scale=1.3]
    \draw[thick, -] (-2, 0) -- (2,0)   node[right] {$H_3$};
    \draw[thick, -] (0,-2) -- (0,2) node[right, above] {$H_4$};
    \draw[thick, -] (-1.5, -1.5) -- (1.5,1.5) node[right] {$H_1$};
    \draw[thick, -] (-1.5,1.5) -- (1.5,-1.5) node[right] {$H_2$};
   \put(40,-25){\tiny{$(+, +, +, +)$}};
     \put(40,25){\tiny{$(+, +, -, +)$}};
          \put(5,55){\tiny{$(-, +, -, +)$}};
       \put(-45,55){\tiny{$(-, +, -, -)$}};
      \put(-80,25){\tiny{$(-, -, -, -)$}};
         \put(-80,-25){\tiny{$(-, -, +, -)$}};
           \put(5,-55){\tiny{$(+, -, +, +)$}};
                  \put(-45,-55){\tiny{$(+, -, +, -)$}};

\end{tikzpicture}
\caption{An arrangement of $4$ lines through the origin in  $\R^2$  as used in Example \ref{ex:dim1} to go from real phase structures to oriented matroids. The necklace ordering of the lines corresponds to the necklace ordering of the facets of $\projFan_M$ defined by the real phase structure in Figure \ref{fig:reallineU34}.}
\label{fig:reallinesig}
\end{figure}
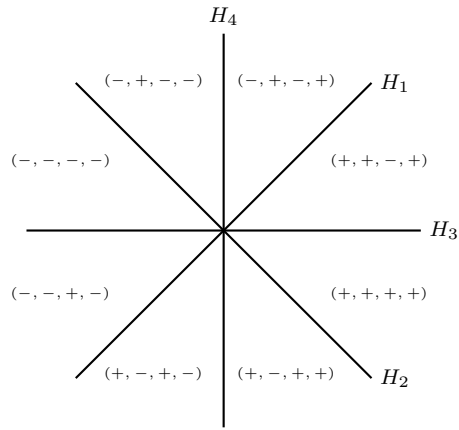

\begin{remark} \label{ex:Reorient}
  Given an oriented matroid $\MMM$ and a subset $S \subset E$, 
	the reorientation of $\MMM$ along $S$ is the oriented matroid 
	$\MMM'$ whose topes are the covectors $r_S(T)$ for any tope $T$
	of $\MMM$. 
	Clearly, in this case $\EEE_{\MMM'}$ is a reorientation of $\EEE_\MMM$ 
	(in the sense of Definition \ref{def:reorientation}) with translation vector $\varepsilon_S$. 
	Hence the map $\mathbf{E}$ from (\ref{eq:orientedtorealstr}) descends to a map modulo reorientations on both sides. 
\end{remark}

\begin{example}\label{ex:dim1}
Here we consider the matroid $M = U_{2, n}$ and describe 
how we can construct an inverse to the map in  (\ref{eq:orientedtorealstr}).
By Example \ref{ex:line}, a real phase structure $\E$ on $\affFan_M$ determines a necklace ordering of $\{1, \dots, n\}$. 
Note that Example \ref{ex:line}  described the projective matroid fan which can be obtained from  $\affFan_M$ by quotienting by $(1, \dots, 1)$. 
 To determine an orientation of $M$ from the real phase structure $\E$, let 
$H_1, \dots, H_n \subset \R^2$ be a collection of pairwise distinct lines passing through the origin in $\R^2$, 
arranged so that the clockwise/anticlockwise
appearance of these lines when making a turn around the origin defines the same  necklace ordering 
on $\{1, \dots, n\}$ as the real phase structure $\E$. 
Two  lines $H_i, H_j$ border a chamber of $\R^2$ if and only if $\E(\sigma_i) \cap \E(\sigma_j) \neq \emptyset$. Moreover, if this is satisfied then $H_i, H_j$ border exactly two chambers which are related by the antipodal map. 
If the intersection of the two affine spaces $\E(\sigma_i)$ and $\E(\sigma_j)$ is non-empty, then it consists of two points $\varepsilon, \varepsilon' = \varepsilon + (1, \dots, 1)$. 
Assigning $(-1)^{\varepsilon} $ to one of the chambers bordered by $H_i, H_j$ and $(-1)^{\varepsilon'} $ to the other  chamber determines an orientation of $U_{2,n}$.  
See Figure \ref{fig:reallinesig}  for the example of the $4$ lines in $\R^2$ corresponding to the real phase structure on the projective fan of $U_{2,4}$ from Figure \ref{fig:reallineU34}.
This procedure associates to any real phase structure on $\Sigma_M$, an  orientation of $M$ and it is easily checked that it 
provides an inverse map to $\mathbf{E}$ from (\ref{eq:orientedtorealstr}). 
\end{example}

\begin{example}\label{ex:codim1}
Here we consider the matroid $M = U_{n-1, n}$. 
Orientations of $M$ can be constructed from generic real hyperplane arrangements like in Example \ref{RealHyperplanes}. 
Moreover, in Example \ref{ex:hyperplane} we showed that $\Sigma_M$ carries a unique real phase structure up to reorientation. 
By Remark \ref{ex:Reorient}, the map  (\ref{eq:orientedtorealstr}) is surjective and hence bijective. 
An analogous discussion shows that (\ref{eq:orientedtorealstr}) is bijective for $M = U_{1,n}$. 
\end{example}

Our main theorem is to prove the equivalence of real phase structures on matroid fans and matroid orientations.
In the more general framework of matroids over hyperfields, oriented matroids are identical to matroids over the sign hyperfield $\mathbb{S} = \{0,+,-\}$ \cite{BakerHyperfields}. 
More precisely, a chirotope describing an oriented matroid can be interpreted as Grassmann-Plücker function
on $E^r$ with values in $\mathbb{S}$. Correspondingly, the affine subspaces  provided by real phase structures
live in a vector space over $\Z_2 \cong \mathbb{S}^*$. For a general hyperfield $\mathbb{H}$ the non-zero elements form a group $\mathbb{H}^*$ under multiplication and we can exchange the role of affine subspaces of 
$\Z_2^n$  for  cosets of subgroups of $(\mathbb{H}^*)^n$.

\begin{ques}\label{ques:hyperfields}
Can matroids over a general hyperfield $\mathbb{H}$ be equivalently formulated by specifying cosets of the group $(\mathbb{H}^*)^n$ on top dimensional faces of a matroid fan $\Sigma \subseteq \R^n$?
\end{ques}

We finish this subsection by showing that the operations of deletion and contraction
on both oriented matroids and real phase structures commute with the map 
$\mathbf{E}$ from Equation (\ref{eq:orientedtorealstr}). We will use the more convenient notation of $\mathcal{E}_{\mathcal{M}} $ to denote the real phase structure $\mathbf{E}(\mathcal{M})$.

\begin{proposition} \label{prop:CompOrientedReal}
  Let $\MMM$ be an oriented matroid on $E$ and $i \in E$. Then
	\[ 
	  \EEE_{\MMM/i} = \EEE_\MMM/i \quad \text{ and } \quad \EEE_{\MMM \backslash i} = \EEE_\MMM \backslash i.
	\]

\end{proposition}

\begin{proof}
  We set $\EEE = \EEE_\MMM$ and denote by $\MMM'$, $\EEE'$, $M'$ the contraction or deletion by $i$, respectively. 
	It suffices to show that $\EEE'(\sigma) \subseteq \EEE_\MMM'(\sigma)$ for every facet $\sigma$ of
	$\Sigma_{M'}$, because both sets are affine subspace of equal dimension. 
By  \cite[Proposition 3.7.11]{bjorner} the covectors of $\MMM\backslash i$ are the projections of the covectors of $\MMM$ under $p_i$. 
	Analogously, the covectors of $\MMM/ i$ are the projections under $p_i$ of the covectors of $\MMM$ with $X_i = 0$
	(hence they are zero on $\cl(i)$). 
	Comparing with Definitions \ref{InducedRealStructures} and \ref{def:OrientedToRealPhase}, the inclusion then follows.
\end{proof}

\subsection{Real subfans and oriented matroid quotients}
\label{subsec:orientedmatroidquotients}

In this subsection, we introduce the concept of \emph{real subfans} and study 
their relationship to  oriented matroid quotients. 
Besides the intrinsic importance of
these constructions, they will be used in the following subsection in the  induction step in the proof of  Theorem \ref{thm:orientedphase}.

Given a real phase structure $\EEE$ on a polyhedral fan $\Sigma$, we extend the definition
of $\EEE$ to non-maximal cones $\tau \in \Sigma$ by setting
\[
 \EEE(\tau) = \bigcup_{\substack{\sigma \text{ facet} \\ \tau \subset \sigma}} \EEE(\sigma).
\]
As an example, note that if $\EEE = \EEE_\MMM$ is induced by an oriented matroid $\MMM$, 
the description of $\EEE(\sigma_\FFF)$ as the set of elements $\varepsilon$
such that $(-1)^\varepsilon$ is adjacent to $\FFF$ extends to non-maximal cones $\sigma_\FFF$, that is, 
\begin{equation} \label{eq:topesmatroids} 
  \E_{\M}(\sigma_{\mathcal{F}} ) = \{ \varepsilon \ | \ (-1)^{\varepsilon} \in \mathcal{T}(\F)\} \subseteq  \Z_2^E.
\end{equation}
The set above no longer has the structure of an affine subspace of $\Z_2^E$ when $\sigma_{\F}$ is not a top dimensional face of the fan. 
Indeed, recall the case of the projective matroid fan  $\projFan_M$ for the uniform matroid $M = U_{n,n+1}$ from Example \ref{ex:hyperplane}. If $\tau$ is the vertex of the fan then the set $\EEE(\tau)$ has size $2^n-1$, hence cannot be an affine space over $\Z_2$. 

\begin{definition}
Given a real phase structure $\EEE'$ on a fan $\Sigma'$, we say that $(\Sigma', \EEE')$
is a \emph{real subfan} of $(\Sigma, \EEE)$ if  $\Sigma' \subseteq \Sigma$ and  for any  $\tau \in \Sigma'$  we have
$\EEE'(\tau) \subseteq \EEE(\tau)$. 
\end{definition}

Real subfans occur naturally when considering the pair of contraction
and deletion of a real phase structure along the some subset.

\begin{lemma} \label{lem:subfans}
  Let $M$ be a matroid on the ground set $E$ and let $\EEE$ be real phase structure on $\Sigma_M$. 
	Let $F$ be a flat of $M$. 
	Then $(\Sigma_{M/F}, \EEE/F)$ is a real subfan of $(\Sigma_{M\backslash F}, \EEE \backslash F)$. 
\end{lemma}

\begin{proof}
	By definition of contraction and deletion for matroids, all flats of  $M/F$ are also flats of $M\backslash F$. In the language of matroids 
	the contraction $M/F$ is a quotient of  the deletion $M\backslash F$. This implies $\Sigma_{M/F} \subseteq \Sigma_{M\backslash F}$, see \cite[Lemma 3.1]{FR-DiagonalTropicalMatroid} or  \cite[Lemma 2.21]{ShawIntMat}.
	For the inclusions $\EEE/F (\tau) \subseteq \EEE \backslash F(\tau)$, by recursion we may reduce to the case where $M$ is loopfree, $\rk(F) = 1$ and $\tau$ a facet of 
	$\Sigma_{M/F}$. 
	
	Let $\{\emptyset \subsetneq F_1 \subsetneq \dots \subsetneq F_l \subsetneq F_{l+1} = E \setminus F\}$ be the chain of flats in $M/F$ 
	corresponding to the face $\tau$. Then $\sigma_0 := p_{F}^{\diamond} (\tau)$ is given by 
	$\FFF_0 = \{\emptyset \subseteq F \subsetneq F_1 \sqcup F \subsetneq \dots \subsetneq F_l \sqcup F \subsetneq E\}$.
	Fix an element $\varepsilon \in \EEE(\sigma_0)$. 
	Removing $F$ from this chain gives rise to a codimension one face $\tau_0$ of $\sigma_0$ such that
	$p_F(\tau_0) = \tau$. By condition (2) of Definition \ref{def:realstructure} there exists
	another facet $\sigma_1$ adjacent to $\tau_0$ such that $\varepsilon \in \EEE(\sigma_1)$. 

	We now distinguish two cases: If $p_F(\sigma_1) \supsetneq \tau$ we are done, since we proved that
	$p_F(\varepsilon) \in \EEE \backslash F (p_F(\sigma_1)) \subseteq \EEE \backslash F (\tau)$. 
	If $p_F(\sigma_1) = \tau$, this implies that $F_1$ is a flat in $M$ and $\sigma_1$ is given by 
	$\FFF_1 = \{\emptyset \subseteq F_1 \subsetneq F_1 \sqcup F \subsetneq \dots \subsetneq F_l \sqcup F \subsetneq E\}$.
	In this case we may follow a similar procedure as above, this time removing the flat $F_1 \sqcup F$ to obtain a face $\tau_1$ of 
	codimension one  of $\sigma_1$ in $\Sigma_M$. Notice that $p_F(\tau_1) = \tau$.   Repeating the argument above we have another facet $\sigma_2$ of $\Sigma_M$ 
	adjacent to $\tau_1$ such that $\varepsilon \in \EEE(\sigma_2)$. 
	Once again, if $p_F(\sigma_2) \supsetneq \tau$ we are done. Otherwise 
	the facet $\sigma_2$ corresponds to a flag $\FFF_2 = \{\emptyset \subseteq F_1 \subsetneq F_2 \subsetneq F_2 \sqcup F \subsetneq \dots \subsetneq F_l \sqcup F \subsetneq E\}$.
	We see that we can continue to repeat the procedure above finding facets $\sigma_k$ of $\Sigma_M$ until  $p(\sigma_k) \supsetneq \tau$, unless it is the case that   
	all $F_i$, and in particular $E \setminus F$, are flats of $M$. This  special case is equivalent to 
	$F$ being a connected component of $M$.
	 In this case, $(\Sigma_{M/F}, \EEE/F) = (\Sigma_{M\backslash F}, \EEE \backslash F)$
	and the statement is trivial. 	
\end{proof}

We denote by $\MMM_1$ and $\MMM_2$ two oriented matroids on the same ground set $E$. For $i=1,2$, 
we denote by $M_i$, $\Sigma_i$, $\EEE_i$ the underlying matroids, associated matroid fans,
and associated real phase structures, respectively. 

We call $\MMM_1$ a \emph{quotient} of $\MMM_2$ if all covectors of $\MMM_1$ are covectors of $\MMM_2$,
see \cite[Section 7.7]{bjorner}. In this case, $M_1$ is a quotient of $M_2$, that is, all flats of $M_1$ 
are flats of $M_2$, see \cite[Corollary 7.7.3]{bjorner}.

\begin{proposition} \label{prop:quotientsubfan}
Let $\MMM_1$ and $\MMM_2$ be two loopfree oriented matroids on $E$.
  Then $\MMM_1$ is a quotient of $\MMM_2$ if and only if $(\Sigma_1, \EEE_1)$ is 
	a real subfan of $(\Sigma_2, \EEE_2)$.
\end{proposition}

\begin{proof}
  In general, given a real phase structure $\EEE_\MMM$ induced by an oriented matroid $\MMM$, 
  the inclusion of matroid fans $\Sigma_1 \subseteq \Sigma_2$ is equivalent to 
	$M_1$ being a (non-oriented) quotient of $M_2$, that is, the flats of $M_1$ are flats of $M_2$, see \cite{FR-DiagonalTropicalMatroid} or  \cite{ShawIntMat}. 
	In particular, if $\MMM_1$ is a quotient of $\MMM_2$, then $\Sigma_1 \subseteq \Sigma_2$. 
	Together with applying Equation (\ref{eq:topesmatroids}) to $\EEE_1(\tau)$ and $\EEE_2(\tau)$,
	the ``only if'' direction follows. 
	
	For the ``if'' direction, let $X$ be a covector of $\MMM_1$. Let $F = \phi(X)$ denote the underlying
	flat of $M_1$. Again, since $\Sigma_1 \subseteq \Sigma_2$, the flat $F$ is also a flat of $M_2$. 
	We consider the chain $\FFF = \{\emptyset \subsetneq F \subsetneq E\}$ and the corresponding ray
	$\sigma_\FFF$. Let $T$ be a tope of $\MMM_1$ such that $X \leq T$. It follows that $T = (-1)^\varepsilon$ with $\varepsilon \in \EEE_1(\sigma_\FFF)$,
	hence by assumption $\varepsilon \in \EEE_2(\sigma_\FFF)$. This, however, implies that $X = T\setminus F$ is a covector of $\MMM_2$, and we are done. 
\end{proof}

\begin{example}\label{ex:quotientU34}
For $M = U_{3, 4}$, the choice of a real phase structure on $\Sigma_M \subset \R^4$ is equivalent to the choice of $\varepsilon \in \Z_2^4$ such that $ \varepsilon , \varepsilon' \not \in \cup_{\sigma} \EEE(\sigma)$, where $\varepsilon' = \varepsilon + (1, \dots, 1)$.

The uniform matroid $M' = U_{2, 4}$ is an ordinary matroid quotient of $U_{3, 4}$.
Given a real phase structure $\EEE$ on $M$, there are $12$ real phase structures on $M'$ such that $\Sigma_{M'}$ 
equipped with one of these real phase structures produces a real subfan of $(\Sigma_M,\EEE)$. 
We now describe them. For every $i \in \{1, \dots, 4\}$ consider the chain of flats 
$$
\mathcal{F}_i :=\left\lbrace \emptyset \subset \left\lbrace i \right\rbrace \subset \left\lbrace 1,2,3,4 \right\rbrace \right\rbrace.
$$ 
The corresponding cone  $\sigma_i$ is in both $\Sigma_M$ and $\Sigma_{M'}$.  The set $\EEE(\sigma_i)$ consists of three affine spaces of dimension $3$ in $\Z_2^4$ which form a necklace arrangement. Moreover, the complement $\Z_2^4 \backslash \EEE(\sigma_i)$
is the unique affine space of dimension two parallel to $T_{\Z_2}(\sigma_i)$
containing the points $\varepsilon, \varepsilon'$.

Contained in $\EEE(\sigma_i)$ there are precisely $3$ affine spaces of dimension $2$ parallel to $T_{\Z_2}(\sigma_i)$. These three affine spaces arise as the intersections of the $3$-dimensional affine spaces in the necklace arrangement at $\sigma_i$.
It follows from this that if $\EEE'(0) \subset \EEE(0)$, then $(\Sigma_{M'}, \EEE') $ is a real subfan of $(\Sigma_M, \EEE)$. 
This is because if  $\EEE'(0) \subset \EEE(0)$ then  for each face $\sigma_i$ the affine space $\EEE'(\sigma_i)$ cannot be equal to the $2$-dimensional affine space $\Z_2^4 \backslash \EEE(\sigma_i)$, and hence $\EEE'(\sigma_i) \subset \EEE(\sigma_i)$.

Following the description from Example \ref{ex:line}, there are a total of $24$ real phase structures on $\Sigma_{M'}$ given by combining the choice of $3$ necklace orderings and a choice of the points in the intersection $\EEE'(\sigma_{i_1}) \cap \EEE'(\sigma_{i_2})$ where $i_1$ and $i_2$ are consecutive faces in the necklace ordering. Of the $24$ real phase structures there are exactly $12$ which contain  $\varepsilon, \varepsilon'$ and $12$ which do not contain them.  Therefore, for  $M' = U_{2, 4}$ there are a total of $12$ real subfans $(\Sigma_{M'}, \EEE')$ of  
$(\Sigma_{M}, \EEE)$ for a fixed real phase structure $\EEE$.

The uniform matroid $M'' = U_{1, 4}$ is also an ordinary matroid quotient of $U_{3, 4}$. The matroid fan of $M'' = U_{1, 4}$ is just a point which we denote by $0$. A real phase structure $\EEE'$ on $\Sigma_{M''}$ produces a real subfan of $( \Sigma_M, \EEE)$ if and only if $\EEE''(0) \neq \{\varepsilon, \varepsilon'\}$.

\end{example}

Following our main theorem, the containment $\EEE'(0) \subset \EEE(0)$ corresponds to containment of topes of the corresponding oriented matroids.  This relation between oriented matroids corresponds to weak maps, \cite[Proposition 7.7.5]{bjorner}. 
The condition of being a real phase subfan is not always equivalent to having $\EEE'(0 ) \subset \EEE(0)$ as the next example shows. 

\begin{example}\label{ex:realsubfan}
Consider again the uniform matroid $ M = U_{3, 4}$. 
The rank $2$ matroid $N$ on $E = \{1, \dots, 4\}$ where $1, 2$ are parallel and $3, 4$ are parallel is also an ordinary matroid quotient of $U_{3, 4}$. The matroid fan of $N$ is an affine space of dimension $2$ in $\R^4$ and hence there are $4$ possible real phase structures on $\Sigma_N$. However, only one of these possible real phase structures produces a real subfan of $(\Sigma_M, \EEE)$. 
Indeed, if $\rho_1$ and $\rho_2$ denote the two half spaces which are top dimensional cones of $\Sigma_M$, then $\EEE(\rho_1)$ and $\EEE(\rho_2)$ are transversely intersecting affine subpaces of dimension $3$ in $\Z_2^4,$
and hence their intersection $\EEE(\rho_1) \cap \EEE(\rho_2)$ gives the unique real phase structure on $\Sigma_N$ yielding a real subfan. 
Yet three out of the four real phase structures on $\Sigma_{N}$ satisfy $\EEE'(0) \subset \EEE(0)$ and hence correspond to weak maps of oriented matroids. 
\end{example}

\subsection{The proof of Theorem \ref{thm:orientedphase}} \label{subsec:equiv}

In this section we prove Theorem \ref{thm:orientedphase}. The idea of the proof is a follows: We use a double induction argument 
on the rank and corank of $M$. 
The deletion and contraction of $M$ by a single element decreases by one the corank and the rank, respectively.  
Then the  contraction matroid is a corank $1$ matroid quotient of the deletion matroid. 
The crucial ingredient in the induction step on the oriented matroid side  is the positive 
answer to the \emph{factorization problem} in corank $1$. We start by recalling the related facts.

Let $\MMM_1$ be a quotient of $\MMM_2$. The factorization problem asks the question whether
there exists an oriented matroid $\MMM$ on a larger ground set $E' \supset E$ such that 
$\rk_\MMM(E^c) = \rk(M_2) - \rk(M_1)$, 
$\MMM_1 = \MMM/E^c$ and $\MMM_2 = \MMM \backslash E^c = \MMM|E$.

Interestingly, the general answer to this question is no, see \cite{RichterGebert93}.
For our purposes, however, it sufficient to consider $\rk(M_2) - \rk(M_1) = 1$,
in which case the answer is positive
\cite{RichterGebertZiegler94}. Here we present a slight generalization of the statement, which allows for  parallel elements. 

\begin{lemma} \label{lem:1ext}
  Let $E' = E \sqcup F$ be a finite set. Let $\MMM_1$ and $\MMM_2$ be oriented matroids on $E$
	such that $\MMM_1$ is a quotient of $\MMM_2$ and $\rk(\MMM_2) - \rk(\MMM_1) = 1$.
	Then there exists an oriented matroid $\MMM$ on $E'$ such that 
	$r_\MMM(F) = 1$, $F$ contains no loops of $\MMM$, $\MMM / F = \MMM_1$ and $\MMM \backslash F = \MMM_2$.
	Moreover, the oriented matroid $\MMM$ is unique up to reorientation of elements in $F$. 
\end{lemma}

\begin{proof}
  Let $e$ be an element of $F$. By \cite[Theorem 4.1]{RichterGebertZiegler94},
	there exists an oriented matroid $\MMM'_2$ on $E \sqcup \{e\}$ of rank $\rk(\MMM_2)$ such that
	$\MMM'_2/e = \MMM_1$ and $\MMM'_2\backslash e = \MMM_2$. 
	It is easy to check that $\MMM'_1 = \MMM_1 \oplus U_{0,1}$ 
	(that is, adding $e$ as a loop) is a quotient of $\MMM'_2$. 
	We can now apply the theorem again to $\MMM'_1$ and $\MMM'_2$
	in order to add another element $e'$ of $F$. 
	Repeating this procedure for the remaining elements in $F$, we obtain an oriented matroid
	$\MMM$ on $E'$ which satisfies the properties of the statement. 
	
	To prove uniqueness, we first note that the underlying matroid $M = \underline{\MMM}$
	is clearly unique by \cite{NeilWhiteTheoryOfMatroids}. To 
	check that the orientation is unique up to reorienting elements in $F$, 
	we consider the circuit description. Note that the number of elements in $F$ contained 
	in a (non-oriented) circuit $C$ of $M$ can be $0$, $1$ or $2$. 
	In the first case, $C$ is a circuit of $M\backslash F$ and hence its orientation is determined by 
	$\MMM_1$. In the second case, $C$ is a circuit in $M/F$ and hence its orientation is determined by $\MMM_1$. 
	Finally, in the last case, since $F$ is a flat of rank $1$, $C$ has to be a (two element) subset 
	of $F$. Clearly, two orientations of $M$ that only differ by the last type of circuits 
	can be obtained from each other by reorientation of $F$. 
	%
\end{proof}

\begin{lemma} \label{lem:RealStructureHiggs}
  Let $M$ be a matroid on $E$ and let $F$ be a flat of rank $1$. 
	Let $\EEE$ and $\EEE'$ be two real phase structures on $\Sigma_M$ such that
	$\EEE/F = \EEE'/F$ and $\EEE \backslash F = \EEE' \backslash F$. 
	Then $\EEE$ and $\EEE'$ agree up to reorientation by an element 
	$\varepsilon$ in the kernel of $p_F$. 
\end{lemma}

\begin{proof}
  Note that the affine matroid fan of $M|F \oplus M/F$
	is the cartesian product of the affine matroid fans of
	its summands, see \cite[Lemma 2.1]{FR-DiagonalTropicalMatroid}. 
	Since $M|F \cong U_{1, |F|}$, its matroid fan
	is $\R(1, \dots, 1) \subset R^F$. 
	Hence $\EEE / F$ determines the 
	real phase  structure on the matroid fan of $M|F \oplus M/F$ induced by $\EEE$,
	up to reorientation along the kernel of $p_F$. 
  In particular, if $F$ is a connected component of $M$, then $M = M|F \oplus M/F$ and 
	the claim follows. 
	
	
	Let us now assume that $F$ is not a connected component of $M$. We pick a fixed 
	facet $\overline{\sigma}$ of $\Sigma_{M \backslash F}$ and set 
	$\sigma := p_F^*(\overline{\sigma})$.  In particular, $\dim(p_F(\sigma)) = \dim(\sigma)$
	and we will refer to such facets as \emph{non-contracted} in the remainder of this proof. Facets not satisfying this are called contracted. 
		
	Since $p_F(\EEE(\sigma)) = \EEE \backslash F( \overline{\sigma})
	= \EEE' \backslash F(\overline{\sigma}) = p_F(\EEE'(\sigma))$, for a non-contracted face $\sigma$, there 
	exists an element $\varepsilon$ in the kernel of $p_F$ such that 
	$\EEE(\sigma) = \EEE'(\sigma) + \varepsilon$. 
	Hence after reorienting $\EEE'$ by $\varepsilon$ we may assume 
	$\EEE(\sigma) = \EEE'(\sigma)$. 
	Assume that a facet  $\sigma'$ is non-contracted. 
	Since $\Sigma_{M \setminus F}$ is connected in codimension one, 
	we can connect $\sigma$ and $\sigma'$ by a sequence of non-contracted facets
	such that successive pairs intersect in codimension one. 	
	Therefore, it is sufficient to show that given a codimension one face $\tau$ 
	of $\sigma$, the data of the  fixed affine space $\EEE(\sigma)$ together with the real phase structures $\EEE /F $  and $\EEE \backslash F$ determine $\EEE$ for all facets adjacent to $\tau$. Hence the real phase structure is unique up to reorientation. 
	By the first paragraph of the proof,  this also determines $\EEE(\sigma')$ for all
	contracted facets $\sigma' \in \Sigma_{M|F \oplus M/F}$. 
		
	Let us formulate the given data  in terms of the necklace arrangement around $\tau$. 
	Knowing $\EEE \backslash F$ is equivalent to knowing the projection of
	the necklace arrangement under $p_F$.
	 If one of the facets adjacent to $\tau$ is contracted, 
	this is equivalent to one of the affine spaces in the  necklace arrangement being contracted to a subspace of dimension one less under $p_F$. In this case, the image of the projection is the intersection of two of the affine subspaces in the corresponding necklace arrangement for $\EEE \backslash F$. The real phase structure $\EEE / F$ tells us  which intersection of affine spaces  in $\EEE \backslash F$ the contracted face maps to.
	This data together with the non-contracted affine space $\EEE(\sigma)$  of the necklace arrangement  uniquely determines the arrangement. 
	If none of the facets adjacent to $\tau$ are contracted, then knowing the projection of the necklace arrangement under $p_F$ and one of the affine space of the necklace arrangement completely determines the necklace arrangement.
	This finishes the proof. 
\end{proof}

\begin{example}\label{ex:realModification}
In this example, we illustrate the main step in the proof of Lemma \ref{lem:RealStructureHiggs} for the matroid 
$M = U_{3, 5}$ and $F = \{5\}$. 
Then $M \backslash 5 = U_{3,4}$ and $M / 5 = U_{2, 4}$. 
Consider the real phase structure $\EEE_{\backslash} $ on $\Sigma_{M \backslash 5}$ with $(0, \dots, 0), (1, \dots, 1) \not \in \EEE_{\backslash}(\sigma)$ for any face $\sigma$ of $\Sigma_{M \backslash 5}$. 

 Let $\tau $ be the codimension one face of $\Sigma_{M}$ corresponding to the chain of flats 
$$\mathcal{F}_1:=\left\lbrace \emptyset \subset \left\lbrace 1 \right\rbrace \subset  \left\lbrace 1,2,3,4, 5\right\rbrace \right\rbrace.$$ 
Let $\sigma_2$ be the facet of $\Sigma_{M}$ corresponding to the chain of flats
$$
\mathcal{F}_{12}:=\left\lbrace \emptyset \subset \left\lbrace 1 \right\rbrace \subset \{1, 2\}\subset \left\lbrace 1,2,3,4, 5\right\rbrace \right\rbrace.$$

Equip  $\Sigma_{M/ 5}$ with a real phase structure $\EEE_/$ having $\EEE_/ (p_5(\tau))  = (0,0, 0, 1) + \langle (1, 0,0,0), (1,1,1,1) \rangle$. 
Notice that then $\EEE_/(p_5(\tau))) \subset \EEE_{\backslash}(p_5(\tau))$. 
Hence $\EEE_/$ can be completed in such a way as to give rise to a  real subfans of $(\Sigma_{M \backslash 5}, E_{\backslash})$,
 see Example \ref{ex:realsubfan}. 
Suppose that  $\EEE$ is a real phase structure on $\Sigma_M$ such that $\EEE \backslash 5 = \EEE_{\backslash}$ and $\EEE / 5 = \EEE_/$.  
 Since we consider real phase structures on $\Sigma_{{M}}$ up to reorientation we are free to fix the affine space for a facet of this fan. Let us set $$\EEE(\sigma_2) = (0,0,0,1,0) + \langle(1, 0, 0, 0,0), (0,1,0,0,0), (1,1,1,1,1) \rangle.$$ 
There are three other facets adjacent to $\tau $ in $\Sigma_{{M}}$ namely $\sigma_3, \sigma_4,  \sigma_5$ where $\sigma_i$ corresponds to the chain of flats 
$$
\mathcal{F}_{1i}:=\left\lbrace \emptyset \subset \left\lbrace 1 \right\rbrace \subset \{1, i\}\subset \left\lbrace 1,2,3,4, 5\right\rbrace \right\rbrace.$$
The facet $\sigma_5$ is contracted under $p_5$, namely $p_5(\sigma_5) = p_5(\tau)$. 

  Since there are only three facets adjacent to a codimension one face of $\Sigma_{M \backslash 5}$ the necklace ordering of facets induced by $\EEE_{\backslash} $ at $p_5 (\tau)$ is unique. The above specification of $\EEE_{/} (p_5(\tau))  $ 
means that $\EEE_/ (p_5(\tau))   = \EEE_{\backslash } (p_5(\sigma_2) \cap \EEE_{\backslash} (p_5(\sigma_3))$. 
Then the necklace ordering induced on the facets in $\Sigma_M$ adjacent to $\tau$ induced by $\EEE$  is 
$\sigma_2, \sigma_5, \sigma_3, \sigma_4$.
The affine space ${\EEE}(\sigma_5)$  must be  the preimage of $\EEE_/(p_5(\tau))$ under the map $p_5$.
Since we have also  fixed $\EEE(\sigma_2)$ above, we see that  
$$\EEE(\sigma_5) \cap \EEE(\sigma_2) =
 (0,0,0,1,0) + 
\langle (1, 0, 0, 0,0) ,(1,1,1,1,1) \rangle.$$
 This intersection  together with the necklace ordering determines completely the necklace arrangement about $\tau$ and hence the other affine spaces $\EEE(\sigma_i)$. 
\end{example}

Lemma \ref{lem:RealStructureHiggs} and Lemma \ref{lem:subfans} tell us that deletion and contraction of real phase structures behave 
as expected in relation to oriented matroid quotients and real subfans in Proposition \ref{prop:quotientsubfan}.
Combining this with the existence of rank $1$ extensions proved in Lemma \ref{lem:1ext}, we are able to prove 
Theorem \ref{thm:orientedphase}.

\begin{proof}[of Theorem \ref{thm:orientedphase}]
  We need to show that any real phase structure $\EEE$ on a matroid fan $\Sigma_M$ can 
	be represented as $\EEE = \EEE_\MMM$ for some oriented matroid $\MMM$. 
	We proceed by double induction on rank and corank. 
	The base cases  for the induction are  $U_{0,n}$ and $U_{n,n}$ and they are  trivial.
	
	In the general case, let $F$ be an arbitrary flat of $M$ of rank $1$. 
	Without loss of generality we may assume that $M$ is loopfree. 
	By the induction assumption, the real phase structures $\EEE / F$ and $\EEE \backslash F$ are represented
	by oriented matroids, say $\MMM_1$ and $\MMM_2$, respectively. 
	By Lemma \ref{lem:subfans} and Proposition \ref{prop:quotientsubfan} we know that
	$\MMM_1$ is a quotient of $\MMM_2$. 
	If $F$ is a connected component of $M$, then the claim follows from the induction assumption applied
	to each connected component. 
	Otherwise, we have $\rk(\MMM_2) - \rk(\MMM_1) = 1$.  By Lemma \ref{lem:1ext}, there
	exists an oriented matroid $\MMM$ on $E$ such that $r_\MMM(F) = 1$, the flat $F$ contains no loops of $\MMM$, 
	$\MMM/F = \MMM_1$ and $\MMM \backslash F = \MMM_2$.
	Since $\underline{\MMM}/F = M/F$ and $\underline{\MMM} \backslash F = M \backslash F$, 
	the uniqueness of this extension for ordinary matroids implies that $\underline{\MMM} = M$,
	see \cite[Proposition 8.3.1]{NeilWhiteTheoryOfMatroids}. 
	It follows that $\EEE' = \EEE_\MMM$ and $\EEE$ are two real phase structures on $\Sigma_M$ whose
	deletion and contraction along $F$ agree. By Lemma \ref{lem:RealStructureHiggs},
	they agree up to reorientation along some $\varepsilon$. Since reorientations
	of oriented matroids and real phase structures are compatible, this shows that the 
	corresponding reorientation of $\MMM$ represents $\EEE$, which proves the claim. 
\end{proof}

\begin{example}\label{example:Fano}
By Theorem \ref{thm:orientedphase} if a matroid is non-orientable then there exists no real phase structure on the corresponding  matroid fan. It is known that the matroids of projective and affine spaces over finite fields are not orientable. 
Moreover, Ziegler constructed an infinite family of minimal non-orientable matroids of rank $3$ \cite{ZieglerNonOr}. Therefore, there are infinitely many matroid fans which do not admit a real phase structure. 
\end{example}

\begin{remark}
The previous discussion shows that  oriented matroid quotients (or real subfans of codimension $1$ of matroid fans) play a special role. 
In general,  if $(\Sigma_1, \EEE_1)$ is a real subfan of $(\Sigma_2, \EEE_2)$, 
we may ask whether this inclusion can be completed 
to a chain of real subfans whose dimensions increase by one in each step. 
Interestingly, there is a  counter-example of  Richter-Gebert  \cite[Corollary 3.4]{RichterGebert93} which shows that this is in general not the case. 
In the language of real phase structures, it gives rise to a pair of real matroid subfans $(\Sigma_1, \EEE_1) \subset (\Sigma_2, \EEE_2)$
with $\dim \Sigma_1 = 1$ and $\dim \Sigma_2 = 3$ such that there exists no real matroid fan
$(\Sigma, \EEE)$ such that $(\Sigma_1, \EEE_1) \subsetneq (\Sigma, \EEE) \subsetneq (\Sigma_2, \EEE_2)$.

This is in contrast to non-oriented matroids (equivalently, matroid fans without real phase structures), 
where the factorization problem can be answered affirmatively and hence such chains always exist \cite[Chapter 8.2]{NeilWhiteTheoryOfMatroids} 
\end{remark}

\subsection{From real phase structures to sign circuits} \label{RealStructureToSignature}
From our main theorem, a real phase structure on a matroid fan is equivalent to specifying an orientation on the underlying matroid and the topes of the oriented matroid are
the points in the real phase structure. Signed circuits are a cryptomorphic description of oriented matroids, and in 
this section, we will describe explicitly how to directly  construct the signed circuit vectors of the oriented matroid arising 
from real phase structure on a matroid fan $\Sigma_M$.

The signed circuits of an oriented matroid consists of a collection of sign vectors $X_C$ and $-X_C$ for every circuit $C $ of the underlying matroid such that $X_C$ has support $C$ and satisfying the signed circuit axioms \cite{bjorner}.
Describing the signed circuits of an oriented matroid $\mathcal{M}$  is  equivalent to choosing for each circuit $C$ of $M$ and  for each pair of elements $i,j \in C$ a sign $\gamma(\MMM)^C_{ij} \in \{\pm 1\}$, such that for all triples   $i,j,k \in C$ we have
\begin{equation}\label{eqn:signaturecheck}
  \gamma(\MMM)^C_{ij} \gamma(\MMM)^C_{jk} \gamma(\MMM)^C_{ik} = +1.
\end{equation}
The signed circuits $X_C$ and $-X_C$ can be recovered  from the assignments $\gamma(\MMM)^C_{ij}$. If $\gamma(\MMM)^C_{ij} = +1$, then $i, j$ have the same sign in $X_C$ and $-X_C$, whereas if $\gamma(\MMM)^C_{ij} =-1$, then $i, j$ have opposite signs in $X_C$ and $-X_C$.

\begin{example} \label{ex:SignaturesForOriented}
  Let $\MMM$ be the oriented matroid on $E = \{1, \dots, n\}$ associated to the (non-zero) linear forms $l_i \colon \R^d \to \R$, $i=1, \dots, n$, see Example \ref{RealHyperplanes}.
	Let $M$ be the underlying matroid. %
	A subset  $C\subset E$ is a circuit of $M$ if there exists a 
	linear relation 
		\[
	  \sum_{i \in C} a_i l_i = 0,
	\]
with  $a_i \neq 0$ for all $i$,  such that this is the unique relation among the linear forms $\{l_i\}_{i\in C}$, up to multiplying by a constant. 
	For any circuit $C$ of $M$ we can assign a signed vector $X_C$ to $C$ by setting $(X_C)_i$ equal to the sign of $a_i$ if $i \in C$ and
	$0$ otherwise.  Multiplying the relation by $-1$ produces $-X_C$. 

	Now, fix $i,j \in C$.
	Let $x$ be a generic point in $\bigcap_{k \in C \setminus \{i,j\}} \{l_k = 0\}$, and consider $X = (\text{sign}(l_s(x)))_s$ the covector of $\MMM$ associated to the cell containing $x$. 
	Note that 
	\[
	  0 = \sum_{k \in C} a_k l_k(x) = a_i l_i(x) + a_j l_j(x). 
	\]
	It follows that the sign of $a_i a_j$ is opposite to the sign of $l_i(x) l_j(x)$, or equivalently,	
	\begin{equation}\label{eq:gammaCircuits}
	\gamma(\MMM)^C_{ij} = (X_C)_i (X_C)_j = - X_i X_j. 
	\end{equation}
	Hence $\gamma(\MMM)^C_{ij}$ can be determined  by comparing entries in certain covectors, with an extra minus sign. Perturbing $x$ slightly to a generic point $x'$,
	we obtain a tope $T = (\text{sign}(l_s(x')))$. Since the values $l_k(x')$, for $k \in C \setminus \{i,j\}$ are arbitrarily small, we still have 
	\[
	  \gamma(\MMM)^C_{ij} = - T_i T_j
	\]
	for such a tope.
In order for a covector to be suitable for computing $\gamma^C_{ij}$ we require that 
	$\text{Supp}(X)^c = \cl(C \setminus \{i,j\})$. 
	Moreover, for a tope to be suitable for computing $\gamma(\MMM)^C_{ij}$ we must have $X < T$. So the suitable topes $T$ for computing the signs $\gamma(\MMM)^C_{ij}$ 
	are the ones that are adjacent to $\cl(C \setminus \{i,j\})$. 
\end{example}

  The previous example  extends directly  to  oriented matroids $\MMM$, as follows. 
	For an oriented matroid $\MMM$, we denote by $\gamma(\MMM)$
	the associated $\gamma$ description  of the signed circuits. 
	For any tope $T$ of $\MMM$ that is adjacent to $\cl(C \setminus \{i,j\})$ 
	we  have 
	\begin{equation} \label{eq:SignFromTopes} 
	  \gamma(\MMM)^C_{ij} = - T_i T_j.
	  	\end{equation} 
		
	\begin{example}
	 Consider the oriented matroid arising from  the complete graph on four vertices $\{1, 2, 3, 4\}$ equipped with the orientation from $i$ to $j$ on the edge $e_{ij}$ when $i < j$. Denote the ground set of $\mathcal{M}$ by $E = \{e_{ij} \}$ and  order the elements of $E$ lexicographically by their indices.  This oriented matroid also arises from the arrangement of real hyperplanes $H_{ij}$ defined by the linear forms $x_j - x_i$ for $1 \leq i < j \leq 4$. The affine fan of the underlying matroid lives in $\R^{6}$ and the projective fan is the cone over the Petersen graph, see \cite[Figure 2]{ArdilaKlivans}. 
	 
	Consider the maximal chain of flats $\FFF := \emptyset \subsetneq \{e_{12}\} \subsetneq \{e_{12},  e_{23}, e_{13}\} \subsetneq E$. We claim that the affine space $\EEE_{\MMM}(\sigma_{\FFF})$ of $\Z_2^6$ is 
	$$ \langle (1, 0, 0, 0, 0, 0),  (0, 1, 0, 1, 0, 0), (0, 0, 1, 0, 1, 1)\rangle. $$  The direction of the affine space is determined by $\FFF$ so it suffices to show that $(0, 0, 0, 0, 0, 0) \in 	\EEE_{\MMM}(\sigma_{\FFF})$.

	 A signed circuit of this graphical arrangement  is for example $C_{123} = (+, -, 0, +, 0, 0)$ coming from the circuit of the graph $\{e_{12},  e_{23}, -e_{13}\}$. 
By the definition of $\gamma(\MMM)$ we have 
	 $ \gamma(\MMM)^C_{e_{13}, e_{23 }} = -1$.	 Therefore, any tope adjacent to the above chain satisfies $T_{e_{13}} T_{e_{23 }} = +1$. 
	 To determine the relations of the other coordinates of a tope vector adjacent to $\FFF$ consider the signed circuits: $C_{124} = (+, 0, -, 0, +, 0)$ and 
$C_{134} = (0, +, -, 0, 0, +)$.  As above we obtain $T_{e_{14}} T_{e_{24 }} = +1$ and $T_{e_{14}} T_{e_{34 }} = +1$. Hence the tope $T = (+, +, +, +, +, +)$ is adjacent to $\FFF$. Therefore, $(0, 0, 0, 0, 0, 0) \in 	\EEE_{\MMM}(\sigma_{\FFF})$ and the claim is proven. 
	   \end{example}	
		
		The following proposition follows directly from the discussion above 
	and the fact that the points in the real phase structure correspond to topes of the oriented matroid associated by Theorem \ref{thm:orientedphase}.

\begin{proposition} \label{prop:realstructureToSignature}
	Let $M$ be matroid with real phase structure $\EEE$ on $\Sigma_M$. 
	The signed circuits with support $C$  of the orientation $\MMM_{\EEE}$ of $M$ arising  from $\EEE$  are described by 
	\[ \gamma(\MMM_{\EEE})^C_{ij}  = (-1)^{\varepsilon_i} (-1)^{\varepsilon_j} \]
	for all $i, j \in C$, where $\varepsilon \in \EEE(\sigma_{\F})$ for $\F$ any flag of flats containing the flat $cl(C\backslash ij)$. 
\end{proposition}

Following our extension of $\EEE$ to arbitrary faces of the matroid fan in Equation (\ref{eq:topesmatroids}) we do not require $\F$ to be a maximal flag of flats.  In particular, we can choose the flag of flats $\F = \{ \emptyset \subsetneq cl(C\backslash ij) \subsetneq E\}$ to determine the signs in the above proposition.